\documentclass[11pt]{article}
\usepackage{amssymb,amsmath,color,amsfonts,float,amscd,amsthm,bm,mathrsfs} % font used for R in Real numbers

\allowdisplaybreaks[4]

\catcode`!=11
\let\!int\int \def\int{\displaystyle\!int}
\let\!lim\lim \def\lim{\displaystyle\!lim}
\let\!sum\sum \def\sum{\displaystyle\!sum}
\let\!sup\sup \def\sup{\displaystyle\!sup}
\let\!inf\inf \def\inf{\displaystyle\!inf}
\let\!cap\cap \def\cap{\displaystyle\!cap}
\let\!max\max \def\max{\displaystyle\!max}
\let\!min\min \def\min{\displaystyle\!min}

\let\oldsection\section
\renewcommand\section{\setcounter{equation}{0}\oldsection}

\newtheorem{theorem}{Theorem}[section]
\newtheorem{lemma}[theorem]{Lemma}

\newtheorem{proposition}[theorem]{Proposition}

\theoremstyle{definition}

\theoremstyle{remark}
\newtheorem{remark}[theorem]{Remark}

\begin{document}

\title{Long time well-posedness of compressible magnetohydrodynamics boundary layer equations in Sobolev space}

\author{Shengxin Li\footnote{School of Mathematical Sciences, Shanghai Jiao Tong University, Shanghai 200240, P. R. China. Email: lishengxin@sjtu.edu.cn}
\and
Feng Xie
\footnote{School of Mathematical Sciences, Shanghai Jiao Tong University, Shanghai 200240, P. R. China.
Email: tzxief@sjtu.edu.cn}
}
\date{}
\maketitle

\begin{abstract}
In this paper we consider the long time well-posedness of solutions to two dimensional compressible magnetohydrodynamics (MHD) boundary layer equations. When the initial data is a small perturbation of a steady solution with size of $\varepsilon$ and the far-field state is also a small perturbation around such a steady solution in Sobolev space, then the lifespan of solutions is proved to be greater than $\varepsilon^{-\frac43}$.
\end{abstract}

%\textit{MSC}:
%35L05, 35L70, 35L72
%
%
%\textit{Key word}:
%\ Global well-posedness; Two dimensional; Nonlinear wave equations;  Without Compact Support; Null condition
\maketitle

%%      ---------------------------------------------------------------------
%%      ------------------- TABLE OF CONTENTS (OPTIONAL) --------------------
%%      ---------------------------------------------------------------------

% \tableofcontents

%%      ---------------------------------------------------------------------
%%      ---------------------------- BODY OF PAPER --------------------------
%%      ---------------------------------------------------------------------

%%      Please input or insert the body of your paper here.

\section{Introduction}
In this paper we  are concerned with the long time well-posedness of  two dimensional compressible non-isentropic magnetohydrodynamics (MHD) boundary layer equations in a domain $\Omega:=\{(x, y)| x\in\mathbb{T},\ y\in\mathbb{R}_+\}$:
\begin{align}\label{1.1}
\begin{cases}
\partial_t\rho+(u\partial_x+v\partial_y)\rho+\rho(\partial_xu+\partial_yv)=0,\\
\rho\left(\partial_tu+(u\partial_x+v\partial_y)u\right)+\partial_x(p+\frac{1}{2} h^2)-(h\partial_x+g\partial_y)h-\mu\partial_y^2u=0,\\
\partial_y(p+\frac{1}{2} h^2)=0,\\
c_v\rho\left(\partial_t\theta+(u\partial_x+v\partial_y)\theta\right)+p(\partial_xu+\partial_yv)-\nu\partial_y^2\theta-\mu(\partial_yu)^2-\kappa(\partial_yh)^2=0,\\
\partial_th+\partial_y(vh-ug)-\kappa\partial_y^2h=0,\\
\partial_xh+\partial_yg=0.
\end{cases}
\end{align}
Here $\rho$ denotes the boundary layer of density, $(u, v)$ the boundary layer of velocity, $(h, g)$ the boundary layer of magnetic field, and $\theta$ the boundary layer of temperature respectively. The pressure $p$ is a function of $\rho$ and $\theta$, which takes the following form
\begin{align}\label{1.2}
p=R\rho\theta,
\end{align}
 where $R>0$ is some constant. The no-slip boundary conditions on velocity and the perfectly conducting boundary conditions on the magnetic field are imposed while the Dirichlet boundary condition is imposed on the temperature.
\begin{align}\label{1.3}
(u, v, \partial_yh, g)|_{y=0}={\bf{0}}, \qquad\theta|_{y=0}=\theta^*(t, x).
\end{align}
The far-field state is denoted by
\begin{align}\label{1.4}
\lim\limits_{y\to\infty}(\rho, u, \theta, h)=(\rho^0, u^0, \theta^0, h^0)(t, x).
\end{align}
where the known functions $\rho^0, u^0, \theta^0, h^0$ are the traces of the density, tangential velocity, temperature and tangential magnetic field of the outflow on the boundary, respectively.

Before proceeding, let us first review some relevant literature on the study of Prandtl boundary layer theory. The famous Prandtl equations was first proposed by Prandtl \cite{P} in 1904 to describe the behavior of viscous fluid under high Reynolds number near the physical boundary. Under the monotonicity condition on the tangential velocity in the vertical direction, Oleinik \cite{O} obtained the local existence of solutions to 2D Prandtl equation by using the Crocoo transformation in 1960s.
One also refer to the classical book \cite{OS} for this result and some other related progress in this field. The Oleinik's local well-posedness theory was reproved by using energy methods directly in \cite{AWXY} and \cite{MW} respectively in Sobolev framework. Xin and Zhang \cite{XZ} obtained a global in time existence of weak solution by imposing an additional favorable condition on the pressure. The above results were extended to three dimensional case in \cite{LWY} and \cite{LWY2}. When the monotonicity assumption was violated, the boundary separation can be observed and the ill-posedness of the Prandtl equation in Sobolev space was thus proved, one can refer to \cite{EE, GD, GN, LWY3, LY} and the reference therein for details.

Without the monotonicity structure condition, it is nature to study Prandtl equation in analytic framework and Gevrey class due to the loss of regularity. In analytic setting, Sammatino and Caflish \cite{SC} established the local well-posedness result of Prandtl equations for the data which is analytic in both $x$ and  $y$. Then Lombardo and his collaborators \cite{LCS} removed the analytic requirement in $y$. The main arguments in these two works rely on the abstract Cauchy-Kowalewskaya theorem. Zhang and Zhang obtained the lifespan of analytic solutions to the classical Prandtl equations with small analytic initial data in \cite{ZZ}. Furthermore, if the initial data is a small analytic perturbation of a Guassian error function,  Ignatova and Vicol \cite{IV} proved an almost global existence for the Prandtl equations. Very recently, Paicu and Zhang \cite{PZ} established the global well-posedness of analytic solutions. On the other hand, in the Gevrey class, G\'ervard-varet and Masmoudi \cite{GM} proved the local well-posedness for 2D Prandtl equation with Gevrey index $7/4$. Later \cite{CWZ} established the well-possdness for the linearized Prandtl equation around a non-monotonic shear flow. When the equation has a non-degenerate critical point, the well-posedness result was obtained in \cite{LY2} with Gevrey index 2, which is optimal in the meaning of \cite{GD}. Furthermore, without any structural condition, \cite{DG} also proved the local well-posedness in Gevrey 2. Recently, this result was extended to three dimensional Prandtl equations in \cite{LNT}.

For plasma, the MHD boundary layer equations was derived from the fundamental MHD equations with no-slip boundary conditions on velocity and perfect conducting boundary conditions on magnetic field \cite{LXY2,LXY3}. And much more abundant boundary layer phenomena are observed due to the coupling effect between the magnetic field and velocity field through the Maxwell equations. At the same time, it also produces more difficult in the mathematical analysis. In general, it is believed that suitable magnetic fields have a stabilizing effect on the boundary layer. For the incompressible magnetohydrodynamics case, \cite{LXY2, LXY3} established the well-posedness of solutions to MHD boundary layer equations and proved the validity of Prandtl boundary layer expansion in Sobolev spaces under the condition that the tangential component of magnetic filed does not degenerate near the physical boundary initially. When the initial data is a small perturbation around the steady solution in analytic space with size of $\varepsilon$, \cite{XY} proved the lifespan of analytic solutions to MHD boundary layer equations is greater than $\varepsilon^{-2^-}$.
This result was extended to global existence in analytic space in \cite{LZ3} and \cite{LX}.  The lifespan of order $\varepsilon^{-2^-}$ was also obtained in Sobolev spaces for incompressible MHD boundary layer equations \cite{CRWZ}. Moreover, under the same assumption as in \cite{LXY2}, \cite{LWXY} also established the local well-posedness of solutions to MHD boundary layer equations without magnetic viscosity.

For the compressible MHD boundary layer equations \eqref{1.1}, the local-in-time well-posedness is first proved in \cite{HLY}.
The aim of this paper is to establish the long time existence of solutions to system (\ref{1.1})-(\ref{1.4}) in Sobolev space. Precisely, if the initial data is a perturbation of a steady solution with size of $\varepsilon$ in Sobolev space, then there exists a unique solution to the initial boundary value problem (\ref{1.1})-(\ref{1.4}) with the lifespan $T_\varepsilon$ being greater than $\varepsilon^{-\frac43}$. To our knowledge, there are very few results on the compressible MHD boundary layer equations, especially for the non-isentropic case. This is one of the main motivations of the study in this paper. Moreover, there is some obviously difference between the compressible and incompressible MHD boundary layer equations in the structure of system and mathematical analysis technique. For example, there is a boundary layer for the density, and thus for the pressure. And the boundary layer function of pressure $p$ and  the boundary layer function of $\frac12h^2$ cancel each other due to $\partial_y(p+\frac{1}{2} h^2)=0$. Thus, these facts will lead to the corresponding difference in analysis between the compressible and incompressible cases. At this moment, the positive lower bound of lifespan is proved to be greater that $\varepsilon^{-\frac43}$. To improve this lower bound of lifespan is also interesting in analysis. In addition, compared with the known results on long time existence or global existence of solutions to the incompressible MHD boundary layer equations in analytic space \cite{LX, LZ3, XY} or Sobolev space \cite{CRWZ}, we do not need to use Gaussian weighted functions any more in performing energy estimates, which makes the analysis include more function classes.

For the related compressible isentropic MHD boundary layer equations, it can be written as follows.
\begin{align}\label{1.5}
\begin{cases}
\partial_t\rho+(u\partial_x+v\partial_y)\rho+\rho(\partial_xu+\partial_yv)=0,\\
\rho\left(\partial_tu+(u\partial_x+v\partial_y)u\right)+\partial_x(p+\frac{1}{2} h^2)-(h\partial_x+g\partial_y)h-\mu\partial_y^2u=0,\\
\partial_y(p+\frac{1}{2} h^2)=0,\\
\partial_th+\partial_y(vh-ug)-\kappa\partial_y^2h=0,\\
\partial_xh+\partial_yg=0.
\end{cases}
\end{align}
The pressure $p$ is a function of $\rho$ which takes the following form.
\begin{align}\label{1.6}
p=\rho^\gamma, \quad \gamma\ge 1.
\end{align}
The boundary condition and the far-field equation are the same as (\ref{1.3})-(\ref{1.4}).

To show the main strategy to establish the long time well-posedness of compressible MHD boundary layer equations, we will first consider the isentropic case \eqref{1.5} with a uniform outer flow. And this idea can be extended to the non-isentropic case \eqref{1.1} with the far-fields being small perturbation of uniform states, which will be discussed in details in section 4.

Consequently, we first focus on the case that $\mu=\kappa=1$ and a uniform outflow $(\rho^0, u^0, h^0)=(1, 0, 1)$ for the initial boundary value problem of compressible isentropic MHD boundary layer equations (\ref{1.5}). Since we consider the system of equations (\ref{1.5}) in the framework of small perturbation around the steady flow $(\rho, u, v, h, g)=(1, 0, 0, 1, 0)$, it is convenient to denote
\[
 \tilde{\rho}=\rho-1,\qquad \tilde{h}=h-1.
\]
Then $(\tilde{\rho}, u, v, \tilde{h}, g)$ solves the following system of equations.
\begin{align}\label{1.7}
\begin{cases}
\partial_t\tilde{\rho}+(u\partial_x+v\partial_y)\tilde{\rho}+\rho(\partial_xu+\partial_yv)=0,\\
\rho\left(\partial_tu+(u\partial_x+v\partial_y)u\right)+\partial_x\left(p+\frac{1}{2} h^2\right)-(h\partial_x+g\partial_y)\tilde{h}-\partial_y^2u=0,\\
\partial_y\left(p+\frac{1}{2} h^2\right)=0,\\
\partial_t\tilde{h}+(u\partial_x+v\partial_y)\tilde{h}+h(\partial_xu+\partial_yv)-(h\partial_x+g\partial_y)u-\partial_y^2\tilde{h}=0,\\
\partial_x\tilde{h}+\partial_yg=0.
\end{cases}
\end{align}
From the third equation in (\ref{1.7}), we have
\begin{align}\label{1.8}
\left(p+\frac 12h^2\right)(t, x, y)\equiv \left((\rho^0)^\gamma+\frac 12(h^0)^2\right)(t, x, 0)=\frac32,
\end{align}
which implies
\begin{align}\label{1.9}
p(t, x, y)=\frac32-\frac 12h^2(t, x, y)>0.
\end{align}
On the other hand, it follows from (\ref{1.6}) that
\begin{align}\label{1.10}
\frac{\partial_i\tilde{\rho}}{\rho}=\frac{\partial_ip}{\gamma p},\quad i=t, x, y,
\end{align}
Substitute the above relationships of (\ref{1.9}) and (\ref{1.10}) into the first equation in (\ref{1.7}), we obtain
\begin{align}\label{1.11}
\partial_xu+\partial_yv=-\frac{\partial_t p+(u\partial_x+v\partial_y)p}{\gamma p}=\frac{h(\partial_t +u\partial_x+v\partial_y)\tilde{h}}{\gamma(\frac32-\frac12 h^2)}.
\end{align}
Combining \eqref{1.11} with the fourth equation in (\ref{1.7}) yields that
\begin{align}\label{1.12}
\partial_xu+\partial_yv=\frac{h(h\partial_x+g\partial_y)u+ h\partial_y^2\tilde{h}}{\gamma(\frac32-\frac12 h^2)+h^2}.
\end{align}
To simplicity of the representation, it is helpful to introduce the following notations.
\begin{align}\label{1.13}
A:=\frac 1{(\frac32-\frac 12h^2)^{1/\gamma}},\quad B:=1-\frac{h^2}{\gamma(\frac32-\frac 12h^2)+h^2},\quad C:=\frac 1{\gamma(\frac32-\frac 12h^2)+h^2}.
\end{align}
Obviously, $A, B, C$ are all positive and bounded provided that $h$ is a small perturbation around $1$. Thus, the system of equations (\ref{1.7}) can be rewritten in the following form
\begin{align}\label{1.14}
\begin{cases}
\partial_tu+(u\partial_x+v\partial_y)u-A(h\partial_x+g\partial_y)\tilde{h}-A\partial_y^2u=0,\\
\partial_t\tilde{h}+(u\partial_x+v\partial_y)\tilde{h}-B(h\partial_x+g\partial_y)u-B\partial_y^2\tilde{h}=0,\\
\partial_xu+\partial_yv=C\left(h(h\partial_x+g\partial_y)u+ h\partial_y^2\tilde{h}\right),\\
\partial_x\tilde{h}+\partial_yg=0,\\
(u, \tilde{h})|_{t=0}=(u_0, \tilde{h}_0)(x, y),
\end{cases}
\end{align}
with the boundary condition
\begin{align}\label{1.15}
(u, v, \partial_y\tilde{h}, g)|_{y=0}=\bf{0},
\end{align}
and the corresponding far-field condition
\begin{align}\label{1.16}
\lim\limits_{y\to\infty}(u, \tilde{h})=(0, 0).
\end{align}
To state the main result, we introduce the following anisotropic Sobolev space $H^{k, l}(\Omega)$. It consists of all functions $f\in L^2(\Omega)$ that for any $k, l\in\mathbb{N}$ satisfies
\[
\|f\|_{H^{k, l}}^2=\sum\limits_{\alpha=0}^{k}\sum\limits_{\beta=0}^{l}\|\partial_\tau^\alpha\partial_y^\beta f\|_{L^2(\Omega)}^2<+\infty,
\]
where $\partial_\tau^{\alpha}=\partial_{t}^{\alpha_1}\partial_{x}^{\alpha_2}$ with $\alpha=(\alpha_1, \alpha_2)$, $|\alpha|=\alpha_1+\alpha_2$. Now it is ready to state the first result.
\begin{theorem}\label{Thm1}(Isentropic case)
For the initial-boundary value problem (\ref{1.14})-(\ref{1.16}), there exists $\varepsilon_0>0$, such that for all $\varepsilon\in(0, \varepsilon_0)$, and the initial data $(u_0, \tilde{h}_0)$ satisfies
\begin{align}\label{1.17}
M\sum\limits_{|\alpha|=0}^{3}\left(\|\sqrt{B_0}\partial_\tau^\alpha u_0\|_{L^2}+\|\sqrt{A_0}\partial_\tau^\alpha\tilde{h}_0\|_{L^2}\right)+\sum\limits_{|\alpha|=0}^{2}\left(\|\sqrt{B_0}\partial_\tau^\alpha\partial_y u_0\|_{L^2}+\|\sqrt{A_0}\partial_\tau^\alpha\partial_y\tilde{h}_0\|_{L^2}\right)\le\varepsilon,
\end{align}
where $M$ is a sufficiently large constant which is independent of $\varepsilon$ and will be determined later, $A_0$ and $B_0$ denote the initial data of $A$ and $B$ respectively defined in \eqref{1.13}. Then there exists a time $T_\varepsilon$, the initial-boundary value problem (\ref{1.14})-(\ref{1.16}) admits a unique solution in the time interval $(0, T_\varepsilon)$. Moreover, $T_\varepsilon$ has a positive lower bound estimate as follows.
\[
T_\varepsilon\ge C\varepsilon^{-\frac43},
\]
where $C$ is a constant independent of $\varepsilon$.
\end{theorem}

For the initial boundary value problem of full compressible MHD boundary layer equations (\ref{1.1})-(\ref{1.4}), without loss of generality, we only concentrate on the case that $c_v=\kappa=\mu=\nu=1$ and the outflow $u^0\equiv 0$. By the same procedure as (1.14), the equation (1.1) can be changed into the following form, also refer to the system (1.14) in \cite{HLY}.
\begin{align}\label{4.1}
\begin{cases}
\partial_tu+(u\partial_x+v\partial_y)u-\frac{R\theta}{P-\frac12h^2}(h\partial_x+g\partial_y)h+\frac{R\theta}{P-\frac12h^2}P_x-\frac{R\theta}{P-\frac12h^2}\partial_y^2u=0,\\
\partial_t\theta+(u\partial_x+v\partial_y)\theta+\frac{a\theta}{Q}h(h\partial_x+g\partial_y)h-\frac{a\theta}{Q}(P_t+P_xu)\\
\qquad\qquad\qquad\qquad\qquad-\frac{a\theta}{Q}\frac{P+\frac12h^2}{P-\frac12h^2}\left[\partial_y^2\theta+(\partial_yu)^2+(\partial_yh)^2\right]-\frac{a\theta}{Q}h\partial_y^2h=0,\\
\partial_th+(u\partial_x+v\partial_y)h-\frac{P-\frac12h^2}{Q}(h\partial_x+g\partial_y)u-\frac{1-a}{Q}(P_t+P_xu)\\
\qquad\qquad\qquad\qquad\qquad-\frac{P-\frac12h^2}{Q}\partial_y^2h+\frac{a}{Q}h\left[\partial_y^2\theta+(\partial_y^2u)^2+(\partial_y^2h)^2\right]=0,\\
\partial_xu+\partial_yv=\frac{1-a}{Q}h[(h\partial_x+g\partial_y)u+\partial_y^2h]+\frac{a}{Q}[\partial_y^2\theta+(\partial_y^2u)^2+(\partial_y^2h)^2]-\frac{1-a}{Q}(P_t+P_xu),\\
\partial_xh+\partial_yg=0,\\
\lim\limits_{y\to\infty}(u, \theta, h)=(u^0, \theta^0, h^0)(t, x, 0)=: (0, \Theta, H)(t, x),\\
(u, v, \partial_yh, g)|_{y=0}={\bf{0}}, \qquad\theta|_{y=0}=\theta^*(t, x),
\end{cases}
\end{align}
with
\begin{align*}
a=\frac{R}{1+R}<1,\quad P(t, x)=p+\frac12h^2,\quad Q(t, x, y)=P+\frac12(1-2a)h^2>0.
\end{align*}
\begin{remark}\label{Rem}
The outflow should satisfy the Bernoulli's law.
\begin{align*}
\begin{cases}
P_x-HH_x=0,\\
\Theta_t-\frac{aP_t\Theta}{P+\frac12(1-2a)H^2}=0,\\
H_t-\frac{P_t H}{(P+\frac12(1-2a)H^2)(R+1)}=0.
\end{cases}
\end{align*}
\end{remark}
Now we it is position to present the main result for the full compressible MHD boundary layer equations (\ref{1.1})-(\ref{1.4}).
\begin{theorem}\label{Thm2} (Non-isentropic case)
For any $\sigma>0$, and a positive bounded function $g(t)\in L^1(\mathbb{R}_+)$, if the outflow and the boundary condition $\theta^*$ satisfy
\begin{align}\label{4.2}
\|(\Theta_x, \theta^*_t, \theta^*_x, P_t, P_x, \Theta-\theta^*)\|_{H^3(\mathbb{T}_x)}\le f(t):=\varepsilon^{1+\sigma}g(t).
\end{align}
Furthermore, there exits $\varepsilon_0>0$, such that for all $\varepsilon\in(0, \varepsilon_0)$, and the initial data $(\tilde{u}_0, \tilde{\theta}_0, \tilde{h}_0)$ which is determined in (\ref{4.12}) satisfies
\begin{align}\label{4.3}
M\sum\limits_{|\alpha|=0}^{3}&\left(\|\sqrt{G_{1, 0}}\partial_\tau^\alpha\tilde{u}_0\|_{L^2}^2+\|\sqrt{G_{2, 0}}\partial_\tau^\alpha\tilde{\theta}_0\|_{L^2}^2+\|\sqrt{G_{3, 0}}\partial_\tau^\alpha(\frac{\tilde{h}^2_0}{2})\|_{L^2}^2\right)\notag\\
&\quad+\sum\limits_{|\alpha|=0}^{2}\left(\|\sqrt{G_{1, 0}}\partial_\tau^\alpha\partial_y\tilde{u}_0\|_{L^2}^2+\|\sqrt{G_{2, 0}}\partial_\tau^\alpha\partial_y\tilde{\theta}_0\|_{L^2}^2+\|\sqrt{G_{3, 0}}\partial_\tau^\alpha\partial_y(\frac{\tilde{h}^2_0}{2})\|_{L^2}^2\right)\le\varepsilon
\end{align}
with $M$ be a sufficiently large constant, which is independent of $\varepsilon$ and will be determined later, $G_{1, 0}, G_{2, 0}, G_{3, 0}$ are the corresponding initial data of $G_1, G_2, G_3$ defined in \eqref{G}. Then, there exists a time $T_\varepsilon$, the initial boundary value problem (\ref{4.1}) admits a unique solution in the time interval $(0, T_\varepsilon)$. Moreover, the lifespan $T_\varepsilon$ has a positive lower bound estimate,
\[
T_\varepsilon\ge C\varepsilon^{-\frac43},
\]
where $C$ is a constant independent of $\varepsilon$.
\end{theorem}
\begin{remark}
In particular, the initial boundary value problem (\ref{4.1}) with a uniform outflow can be regarded as  a special case of Theorem
\ref{Thm2}.
\end{remark}
\begin{remark}
For the isentropic case, we also can establish the similar result as that in Theorem \ref{Thm2} provided that the outflow is also a small perturbation of uniform states.
\end{remark}
\begin{remark}
By Bernoulli's law, we do not need to impose any condition on $H_x, H_t, \Theta_t$ in \eqref{4.2}.
\end{remark}

The rest of this paper is organized as follows. In section 2, we introduce a coordinate transformation on the system (\ref{1.14}) to overcome the difficulty of loss of  regularity. At the same time some elementary lemmas are also given in this section. In section 3, the isentropic compressible MHD boundary layer system (\ref{2.3}) with uniform outflow is studied, and the positive lower bound of lifespan is obtained due to the energy estimates established in this section. The non-isentropic compressible MHD boundary layer system with a general outflow is considered in section 4, the lower positive bound of lifespan of solutions is also achieved by a similar energy estimate arguments.

Some notations used frequently in this paper are introduced below. We use the symbol $A \lesssim B$ to stand for $A \le CB$, where $C$ is a uniform constant which may vary from line to line. $\langle a, b\rangle_{H^{k, 0}}\triangleq\int_{\Omega}\partial_\tau^ka(x, y)\partial_\tau^kb(x, y)\;dxdy$ means the $H^{k, 0}$ inner product of $a$, $b$ on $\Omega\triangleq\mathbb{T}\times\mathbb{R}_+$, and $\langle a, b\rangle\triangleq\int_{\Omega}a(x, y)b(x, y)\;dxdy$ the $L^2$ inner product.

\section{Preliminary and elementary lemmas}
The main difficulty of solvability of Prandtl type boundary layer equations in Sobolev space lies in the loss of regularity with respect to tangential variable. To overcome this difficulty, we will adopt a coordinate transformation in terms of the stream function of magnetic field which is proposed in \cite{LXY2} under the assumption that the tangential magnetic field has a positive lower bound, i.e. $h(t, x, y)\ge c_0>0$. It is noted that such an assumption is automatically satisfied provided that $h(t, x, y)$ is a small perturbation of $1$.

\subsection{Coordinate transformation}
Inspired by \cite{LXY2}, from the divergence free condition of $\partial_xh+\partial_yg=0$, there exists a steam function $\psi(t, x, y)$ such that
\[
h=\partial_y\psi, \quad g=-\partial_x\psi,\quad \psi|_{y=0}=0.
\]
Moreover, from the fourth equation in (\ref{1.7}) and the boundary condition (\ref{1.15}), it is direct to check that $\psi$ satisfies the following equation
\begin{align}\label{2.1}
\partial_t\psi+(u\partial_x+v\partial_y)\psi-\partial_y^2\psi=0.
\end{align}
Under the assumption that the tangential magnetic field $h$ has a positive lower bound, it is valid to introduce the coordinate transformation,
\[
\bar{t}=t,\quad \bar{x}=x,\quad \bar{y}=\psi(t, x, y),
\]
and the new unknown functions
\[
(\hat{u}, \hat{h})(\bar{t}, \bar{x}, \bar{y}):=(u, h)(t, x, y).
\]
Under this new coordinate, the region $\{(t, x, y)|t\in (0, T_\varepsilon), x\in \mathbb{T},  y\in\mathbb{R}_+\}$ is mapped into $\{(\bar{t}, \bar{x}, \bar{y})|\bar{t}\in (0, T_\varepsilon),  \bar{x}\in\mathbb{T},  \bar{y}\in\mathbb{R}_+\}$, and the boundary of $\{ y=0\}$ ($\{ y=+\infty\}$ respectively) becomes the boundary of $\{ \bar{y}=0\}$ ($\{ \bar{y}=+\infty\}$ respectively). Also the equations in (\ref{1.14}) can be transformed into
\begin{align}\label{2.2}
\begin{cases}
\partial_{\bar{t}}u+u\partial_{\bar{x}}u-Ah\partial_{\bar{x}}\tilde{h}+(1-A)h\partial_{\bar{y}}\tilde{h}\partial_{\bar{y}}u-Ah^2\partial_{\bar{y}}^2u=0,\\
\partial_{\bar{t}}\tilde{h}+u\partial_{\bar{x}}\tilde{h}-Bh\partial_{\bar{x}}u+(1-B)h(\partial_{\bar{y}}\tilde{h})^2-Bh^2\partial_{\bar{y}}^2\tilde{h}=0.
\end{cases}
\end{align}
Here and after, we omit all ``hat" for simplicity of representation without causing confuse. Symmetrizing this system \eqref{2.2} and replacing $(\bar{t}, \bar{x}, \bar{y})$ with $(t, x, y)$, we obtain
\begin{align}\label{2.3}
\begin{cases}
B\partial_{t}u+Bu\partial_{x}u-ABh\partial_{x}\tilde{h}+(1-A)Bh\partial_{y}\tilde{h}\partial_{y}u-ABh^2\partial_{y}^2u=0,\\
A\partial_{t}\tilde{h}+Au\partial_{x}\tilde{h}-ABh\partial_{x}u+(1-B)Ah(\partial_{y}\tilde{h})^2-ABh^2\partial_{y}^2\tilde{h}=0,\\
(u, \tilde{h})|_{t=0}=(u_0, \tilde{h}_0)(x, y),
\end{cases}
\end{align}
the corresponding boundary condition
\begin{equation}\label{2.4}
(u, \partial_{y}\tilde{h})|_{y=0}=\bf{0},
\end{equation}
and the far-field state
\begin{align}\label{2.5}
\lim\limits_{y\to +\infty}(u, \tilde{h})=(0, 0).
\end{align}

\subsection{Some notations and elementary lemmas}
In this subsection, we first introduce new anisotropic Sobolev space $H^{k, l}(\Omega)$ in new coordinate with the corresponding norm.
\[
\|f\|_{H^{k, l}}^2=\sum\limits_{|\alpha|=0}^{k}\sum\limits_{\beta=0}^{l}\|\partial_\tau^\alpha\partial_y^\beta f\|
_{L^2(\Omega)}^2<+\infty.
\]
Here, $\partial_\tau^{\alpha}=\partial_{t}^{\alpha_1}\partial_{x}^{\alpha_2}$ with $\alpha=(\alpha_1, \alpha_2)$, $|\alpha|=\alpha_1+\alpha_2$ and $\Omega=\{(x, y)| x\in\mathbb{T},\ y\in\mathbb{R}_+\}$. Then we define the following energy functionals
\begin{align}\label{M}
&E(t)=M\sum\limits_{|\alpha|=0}^{3}\left(\|\sqrt{B}\partial_\tau^\alpha u\|_{L^2}^2+\|\sqrt{A}\partial_\tau^\alpha\tilde{h}\|_{L^2}^2\right)+\sum\limits_{|\alpha|=0}^{2}\left(\|\sqrt{B}\partial_\tau^\alpha\partial_y u\|_{L^2}^2+\|\sqrt{A}\partial_\tau^\alpha\partial_y\tilde{h}\|_{L^2}^2\right),\notag\\
&D(t)=M\|\partial_{y}(u, \tilde{h})\|_{H^{3, 0}}^2+\|\partial_{y}(u, \tilde{h})\|_{H^{2, 1}}.
\end{align}
Here $M$ is a suitably large constant which will be determined later.

From now on, we always assume $(u, \tilde{h})$ is a smooth solution to  (\ref{2.3})-(\ref{2.5}), and there exists a time $T$ such that
\begin{align}\label{2.6}
E(t)\le 8\varepsilon^2
\end{align}
holds for any $t\in[0, T]$.

Below, we introduce a Sobolev-Gagliardo-Nirenberg-Moser type inequality and a Sobolev embedding inequality whose proof also can be found in \cite{LXY2}.
\begin{lemma}\label{Lem2.1}
For any $f, g\in H^{k, 0}$ and $\alpha, \beta\in\mathbb{N}^2$ satisfy $|\alpha|+|\beta|=k$, it holds
\begin{align*}
\|\partial_\tau^\alpha f\partial_\tau^\beta g\|_{L^2}\lesssim \|f\|_{L^\infty}\|g\|_{H^{k, 0}}+\|g\|_{L^\infty}\|f\|_{H^{k, 0}}.
\end{align*}
\end{lemma}

\begin{lemma}\label{Lem2.3}
For any proper $f$, if $\lim\limits_{y\to +\infty}f(x,y)=0$, then the following Gagliardo-Nirenberg inequality holds
\begin{align}\label{2.7}
\|f\|_{L_y^\infty}\le C\|f\|_{L_y^2}^{\frac12}\|\partial_y f\|_{L_y^2}^{\frac12}.
\end{align}
\end{lemma}
\begin{lemma}\label{Lem2.2}
Under the a priori assumption (\ref{2.6}), there exists $\varepsilon_0>0$, such that for any $\varepsilon\in (0, \varepsilon_0)$, it holds
\begin{equation*}
h(t, x, y)\ge \frac 12.
\end{equation*}
\end{lemma}
The proof of Lemma \ref{Lem2.2} is straightforward.

Finally, we present some derivative estimates of the coefficient $A, B$ which will be frequently used in the next section.
\begin{lemma}\label{Lem2.4}
For $A, B$ defined in (\ref{1.13}), we have the following estimates
\begin{align}\label{2.8}
\|\partial_\tau B\|_{H^{2, 0}}+\|\partial_y B\|_{H^{2, 0}}+\|Bu\|_{H^{3, 0}}+\|Bu\|_{L^\infty}\le CE(t)^{\frac12},
\end{align}
\begin{align}\label{2.9}
\|\partial_\tau (Bu)\|_{L^\infty}+\|\partial_\tau B\|_{L^\infty}\le CD(t)^{\frac14}E(t)^{\frac14},
\end{align}
\begin{align}\label{2.10}
\|\partial_y B\|_{L^\infty}+\|\partial_y(Bu)\|_{L^\infty}+\|\partial_y(Bu)\|_{H^{2, 0}}\le CD(t)^{\frac12}.
\end{align}
It is noticed that the above estimates also hold true to replace B with A.
\end{lemma}
The proof is also straightforward, we leave it out here.

\section{Isentropic MHD boundary layer equations}
This section is devoted to deriving the lower bound of lifespan of the classical solution $(u, \tilde{h})$ to the system (\ref{2.3})-(\ref{2.5}). To this end, we will establish the desired energy estimates of solutions, which are listed in the following proposition.
\begin{proposition}\label{Prop3.1}
Under the a priori assumption (\ref{2.6}), for ant $t\in (0, T)$, there exist constants $C, C_0>0$, such that for any $\alpha\in\mathbb{N}^2, \beta\in\mathbb{N}$, it holds
\begin{align}\label{3.1}
\frac d{dt}\sum\limits_{\beta=0}^{1}\sum\limits_{|\alpha|=0}^{3-\beta}(\|\sqrt{B}\partial_\tau^\alpha\partial_y^\beta u\|_{L^2}^2&+\|\sqrt{A}\partial_\tau^\alpha\partial_y^\beta\tilde{h}\|_{L^2}^2)+C_0D(t) \notag\\
&\le CD(t)^{\frac 14}E(t)^{\frac54}+CD(t)^{\frac 12}E(t)+CD(t)E(t)^{\frac14}.
\end{align}
\end{proposition}
The proof will be divided into two parts according to the value of $\beta$, that is, $\beta=0$, and $\beta=1$.

\subsection{Tangential derivative estimates}
The goal of this subsection is to establish $H^{3, 0}$ estimates of the solution $(u, \tilde{h})$ to the initial boundary value problem (\ref{2.3})-(\ref{2.5}).
\begin{proposition}\label{Prop3.2}
Under the a priori assumption (\ref{2.6}), for any $t\in (0, T)$, there exists constants $C, C_1>0$, such that
\begin{align}\label{3.2}
\frac d{dt}\sum\limits_{|\alpha|=0}^{3}(\|\sqrt{B}\partial_\tau^\alpha u\|_{L^2}^2&+\|\sqrt{A}\partial_\tau^\alpha\tilde{h}\|_{L^2}^2)+C_1\|\partial_y(u, \tilde{h})\|_{H^{3, 0}}^2\notag\\
&\le CD(t)^{\frac14}E(t)^{\frac54}+CD(t)^{\frac 12}E(t)+CD(t)E(t)^{\frac12}.
\end{align}
\end{proposition}

\begin{proof}
Take the $H^{3, 0}$ inner product on the equations (\ref{2.3}) with $(u, \tilde{h})$ respectively, we have
\begin{align*}
0=\langle B\partial_t u&,  u\rangle_{H^{3, 0}}+\langle A\partial_t\tilde{h}, \tilde{h}\rangle_{H^{3, 0}}
+\left\langle Bu\partial_x u, u\right\rangle_{H^{3, 0}}+\left\langle Au\partial_x\tilde{h}, \tilde{h}\right\rangle_{H^{3, 0}}\\
&-\left\langle ABh\partial_x\tilde{h}, u\right\rangle_{H^{3, 0}}
-\left\langle ABh\partial_x u, \tilde{h}\right\rangle_{H^{3, 0}}\\
&+\left\langle (1-A)Bh\partial_y\tilde{h}\partial_y u, u\right\rangle_{H^{3, 0}}
+\left\langle (1-B)Ah\partial_y\tilde{h}\partial_y \tilde{h}, \tilde{h}\right\rangle_{H^{3, 0}}\\
&-\left\langle ABh^2\partial_y^2u, u\right\rangle_{H^{3, 0}}
-\left\langle ABh^2\partial_y^2\tilde{h}, \tilde{h}\right\rangle_{H^{3, 0}}=: \sum\limits_{i=1}^{10}I_i.
\end{align*}
Next, we will handle $I_i(i=1, ...,10)$ term by term as follows.
\begin{align}\label{3.3}
I_1&=\sum\limits_{|\alpha|=0}^{3}\left\langle\partial_\tau^\alpha(B\partial_t u), \partial_\tau^\alpha u\right\rangle=\sum\limits_{|\alpha|=0}^{3}\left\langle B\partial_\tau^\alpha\partial_t u, \partial_\tau^\alpha u\right\rangle
+\sum\limits_{|\alpha|=1}^{3}\sum\limits_{|\beta|=1}^{|\alpha|}\left\langle \partial_\tau^\beta B\partial_\tau^{\alpha-\beta}\partial_t u, \partial_\tau^\alpha u\right\rangle\notag\\
&=\frac 12\frac d{dt}\sum\limits_{|\alpha|=0}^{3}\left\|\sqrt{B}\partial_\tau^\alpha u\right\|_{L^2}^2-\frac 12\sum\limits_{|\alpha|=0}^{3}\int_{\Omega} B_t(\partial_\tau^\alpha u)^2+\sum\limits_{|\alpha|=1}^{3}\sum\limits_{|\beta|=1}^{|\alpha|}\left\langle \partial_\tau^\beta B\partial_\tau^{\alpha-\beta}\partial_t u,
\partial_\tau^\alpha u\right\rangle\notag\\
&\ge \frac 12\frac d{dt}\sum\limits_{|\alpha|=0}^{3}\left\|\sqrt{B}\partial_\tau^\alpha u\right\|_{L^2}^2-C\|\partial_\tau B\|_{L^\infty}\|u\|_{H^{3, 0}}^2-C\|\partial_\tau u\|_{L^\infty}\|\partial_\tau B\|_{H^{2, 0}}\|u\|_{H^{3, 0}}\notag\\
&\ge \frac 12\frac d{dt}\sum\limits_{|\alpha|=0}^{3}\left\|\sqrt{B}\partial_\tau^\alpha u\right\|_{L^2}^2-CD(t)^{\frac14}E(t)^{\frac54},
\end{align}
where in the last inequality, Lemma \ref{2.4} is used. Similarly,
\begin{align}\label{3.4}
I_2\ge  \frac 12\frac d{dt}\sum\limits_{|\alpha|=0}^{3}\left\|\sqrt{A}\partial_\tau^\alpha\tilde{h}\right\|_{L^2}^2-CD(t)^{\frac 14}E(t)^{\frac 54}.
\end{align}
And  $I_3$ can be estimated as follows.
\begin{align}\label{3.5}
|I_3|&\le\sum\limits_{|\alpha|=0}^{3}\left|\left\langle\partial_\tau^\alpha(Bu\partial_x u), \partial_\tau^\alpha u\right\rangle\right |\notag\\
&\le\sum\limits_{|\alpha|=0}^{3}\left|\langle Bu\partial_x\partial_\tau^\alpha u, \partial_\tau^\alpha u\rangle\right|
+\sum\limits_{|\alpha|=0}^{3}\sum\limits_{|\beta|=1}^{|\alpha|}\left|\left\langle\partial_\tau^\beta (Bu)\partial_\tau^{\alpha+1-\beta} u, \partial_\tau^\alpha u\right\rangle_{L^2}\right|\notag\\
&=\sum\limits_{|\alpha|=0}^{3}\left|-\frac 12\int_{\Omega} \partial_x (Bu)(\partial_\tau^\alpha u)^2\right|
+\sum\limits_{|\alpha|=0}^{3}\sum\limits_{|\beta|=1}^{|\alpha|}\left|\left\langle\partial_\tau^\beta (Bu)\partial_\tau^{\alpha+1-\beta} u, \partial_\tau^\alpha u\right\rangle_{L^2}\right|\notag\\
&\lesssim\|\partial_\tau(Bu)\|_{L^\infty}\|u\|_{H^{3, 0}}^2+\|\partial_\tau u\|_{L^\infty}\|Bu\|_{H^{3, 0}}\|u\|_{H^{3, 0}}\notag\\
&\le CD(t)^{\frac14}E(t)^{\frac54}.
\end{align}
Along the same line, we also have
\begin{align}\label{3.6}
|I_4|\le CD(t)^{\frac14}E(t)^{\frac54}.
\end{align}
To estimate $I_5+I_6$, we divide $I_5+I_6$ into the following three parts.
\begin{align*}
|I_5+I_6|&\le\sum\limits_{|\alpha|=0}^{3}\left|-\left\langle \partial_\tau^\alpha(ABh\partial_x\tilde{h}), \partial_\tau^\alpha u\right\rangle
-\left\langle\partial_\tau^\alpha(ABh\partial_x u), \partial_\tau^\alpha\tilde{h}\right\rangle\right|\\
&\le\left|-\left\langle ABh\partial_x\partial_\tau^3\tilde{h},\partial_\tau^3u\right\rangle-\left\langle ABh\partial_x\partial_\tau^3u,\partial_\tau^3\tilde{h}\right\rangle\right|\\
&\qquad+\sum\limits_{|\alpha|=1}^{3}\left(\left|\left\langle\partial_\tau^\alpha(ABh)\partial_\tau^{4-\alpha}\tilde{h}, \partial_\tau^3u\right\rangle\right|+\left|\left\langle\partial_\tau^\alpha(ABh)\partial_\tau^{4-\alpha}u, \partial_\tau^3\tilde{h}\right\rangle\right|\right)\\
&\qquad+\sum\limits_{|\alpha|=0}^{2}\sum\limits_{|\beta|=0}^{|\alpha|}\left(\left|\left\langle\partial_\tau^\beta(ABh)\partial_\tau^{\alpha+1-\beta}\tilde{h}, \partial_\tau^\alpha u\right\rangle\right|+\left|\left\langle\partial_\tau^\beta(ABh)\partial_\tau^{\alpha+1-\beta}u, \partial_\tau^\alpha\tilde{h}\right\rangle\right|\right).
\end{align*}
For the first term, integration by parts leads to
\begin{align*}
&\left|-\left\langle ABh\partial_x\partial_\tau^3\tilde{h},\partial_\tau^3u\right\rangle-\left\langle ABh\partial_x\partial_\tau^3u,\partial_\tau^3\tilde{h}\right\rangle\right|\\
=&\left|\left\langle\partial_x(ABh)\partial_\tau^3\tilde{h},\partial_\tau^3u\right\rangle\right|\\
\le& CD(t)^{\frac14}E(t)^{\frac54}.
\end{align*}
And for the remaining two terms, it follows from Lemma \ref{Lem2.1} and \ref{Lem2.4} that
\begin{align*}
&\sum\limits_{|\alpha|=1}^{3}\left(\left|\left\langle\partial_\tau^\alpha(ABh)\partial_\tau^{4-\alpha}\tilde{h}, \partial_\tau^3u\right\rangle\right|+\left|\left\langle\partial_\tau^\alpha(ABh)\partial_\tau^{4-\alpha}u, \partial_\tau^3\tilde{h}\right\rangle\right|\right)\\
&\qquad+\sum\limits_{|\alpha|=0}^{2}\sum\limits_{|\beta|=0}^{|\alpha|}\left(\left|\left\langle\partial_\tau^\beta(ABh)\partial_\tau^{\alpha+1-\beta}\tilde{h}, \partial_\tau^\alpha u\right\rangle\right|+\left|\left\langle\partial_\tau^\beta(ABh)\partial_\tau^{\alpha+1-\beta}u, \partial_\tau^\alpha\tilde{h}\right\rangle\right|\right)\\
&\le CD(t)^{\frac14}E(t)^{\frac54}.
\end{align*}
Collecting the above two inequalities, we get
\begin{align}\label{3.7}
|I_5+I_6|\le CD(t)^{\frac14}E(t)^{\frac54}.
\end{align}
$I_7$ can be estimated in a direct way.
\begin{align*}
|I_7|&=\left|\left\langle (1-A)Bh\partial_y\tilde{h}\partial_y u, u\right\rangle_{H^{3, 0}}\right|\\
&\lesssim\left\|(1-A)Bh)\right\|_{L^\infty}\|\partial_y\tilde{h}\partial_y u\|_{H^{3, 0}}\|u\|_{H^{3, 0}}+\|\partial_y\tilde{h}\partial_y u\|_{L^\infty}\|(1-A)Bh\|_{H^{3, 0}}\|u\|_{H^{3, 0}}.
\end{align*}
From the Gagliardo-Nirenberg inequality,
\begin{align*}
\|\partial_y(u, \tilde{h})\|_{L^\infty}\le C\|\partial_y(u, \tilde{h})\|_{H^{1, 0}}^\frac12\|\partial_y^2(u, \tilde{h})\|_{H^{1, 0}}^\frac12\le CD(t)^{\frac12},
\end{align*}
and by Lemma \ref{Lem2.1}, we have
\begin{align*}
\|\partial_y\tilde{h}\partial_y u\|_{H^{3, 0}}\lesssim\|\partial_y\tilde{h}\|_{L^\infty}\|\partial_y u\|_{H^{3, 0}}+\|\partial_y u\|_{L^\infty}\|\partial_y\tilde{h}\|_{H^{3, 0}}\le CD(t).
\end{align*}
Thus, we arrive at
\begin{align}\label{3.8}
|I_7|\le CD(t)E(t)\le CD(t)E(t)^{\frac12}.
\end{align}
By the similar arguments,
\begin{align}\label{3.9}
|I_8|\le CD(t)E(t)^{\frac12}.
\end{align}
As for $I_9$, we divide it into two parts
\begin{align*}
I_9&=-\sum\limits_{|\alpha|=0}^{3}\left\langle\partial_\tau^\alpha(ABh^2\partial_y^2 u), \partial_\tau^\alpha u\right\rangle\\
&=-\sum\limits_{|\alpha|=0}^{3}\left\langle ABh^2\partial_y^2\partial_\tau^\alpha u, \partial_\tau^\alpha u\right\rangle-\sum\limits_{|\alpha|=0}^{3}\sum\limits_{|\beta|=1}^{|\alpha|}\left\langle\partial_\tau^\beta(ABh^2)\partial_\tau^{\alpha-\beta}\partial_y^2 u, \partial_\tau^\alpha u\right\rangle\\
&=:I_9^1+I_9^2.
\end{align*}
For the first term
\begin{align*}
I_9^1&=\sum\limits_{|\alpha|=0}^{3}\int_{\Omega}ABh^2(\partial_\tau^\alpha\partial_y u)^2+\sum\limits_{|\alpha|=0}^{3}\left\langle\partial_y(ABh^2)\partial_\tau^\alpha\partial_y u, \partial_\tau^\alpha u\right\rangle\\
&\ge C\|\partial_y u\|_{H^{3, 0}}^2-\|\partial_y (ABh^2)\|_{L^\infty}\|\partial_y u\|_{H^{3, 0}}\|u\|_{H^{3, 0}}\\
&\ge C\|\partial_y u\|_{H^{3, 0}}^2-CD(t)E(t)^{\frac12}.
\end{align*}
And for the second term, by Lemma \ref{Lem2.1}, we have
\begin{align*}
|I_9^2|&\lesssim \|\partial_\tau(ABh^2)\|_{L^\infty}\|\partial_y^2u\|_{H^{2, 0}}\|u\|_{H^{3, 0}}+\|\partial_y^2 u\|_{L_\tau^\infty L_y^2}\|\partial_\tau(ABh^2)\|_{H_\tau^2L_y^\infty}\|u\|_{H^{3, 0}}\\
&\le CD(t)^{\frac 12}E(t)+CD(t)E(t)^{\frac 12}.
\end{align*}
As a consequence,
\begin{align}\label{3.10}
I_9\ge C_1\|\partial_y u\|_{H^{3, 0}}^2- CD(t)^{\frac 12}E(t)-CD(t)E(t)^{\frac12}.
\end{align}
In a similar way,
\begin{align}\label{3.11}
I_{10}\ge C_1\|\partial_y\tilde{h}\|_{H^{3, 0}}^2- CD(t)^{\frac 12}E(t)-CD(t)E(t)^{\frac12}.
\end{align}
Finally, collecting all the estimates (\ref{3.3})-(\ref{3.11}) together, we conclude that
\begin{align}\label{3.12}
\frac d{dt}\sum\limits_{|\alpha|=0}^{3}(\|\sqrt{B}\partial_\tau^\alpha u\|_{L^2}^2&+\|\sqrt{A}\partial_\tau^\alpha\tilde{h}\|_{L^2}^2)+C_1\|\partial_y(u, \tilde{h})\|_{H^{3, 0}}^2\notag\\
&\le CD(t)^{\frac14}E(t)^{\frac54}+CD(t)^{\frac 12}E(t)+CD(t)E(t)^{\frac12}.
\end{align}
Consequently, the proof of Proposition \ref{Prop3.2} is done.
\end{proof}

\subsection{Normal derivative estimates}
The subsection is intended to establish $H^{2, 1}$ estimates of the solution $(u, \tilde{h})$ to the initial boundary value problem (\ref{2.3})-(\ref{2.5}).
\begin{proposition}\label{Prop3.3}
Under the a priori assumption (\ref{2.6}), for any $t\in (0, T)$, there exist constants $C, C_2>0$, such that for $M$ large enough defined in (\ref{M}), it holds
\begin{align}\label{3.13}
\frac d{dt}\sum\limits_{|\alpha|=0}^{2}(\|\sqrt{B}\partial_\tau^\alpha\partial_y u\|_{L^2}^2&+\|\sqrt{A}\partial_\tau^\alpha\partial_y\tilde{h}\|_{L^2}^2)+C_2\|\partial_y(u, \tilde{h})\|_{H^{2, 1}}^2\notag\\
&\le CD(t)^{\frac 12}E(t)+CD(t)E(t)^{\frac14}+\frac{C}{M^{\frac14}} D(t).
\end{align}
\end{proposition}

\begin{proof}
Applying $\partial_y$ on equations (\ref{2.3}), and then taking the $H^{2, 0}$ inner product on the resulting equations with $(\partial_y u, \partial_y\tilde{h})$ respectively, we have
\begin{align*}
0=&\left\langle \partial_y(B\partial_t u),  \partial_y u\right\rangle_{H^{2, 0}}+\left\langle \partial_y (A\partial_t\tilde{h}), \partial_y\tilde{h}\right\rangle_{H^{2, 0}}\\
&+\left\langle \partial_y(Bu\partial_x u), \partial_y u\right\rangle_{H^{2, 0}}+\left\langle \partial_y(Au\partial_x\tilde{h}), \partial_y\tilde{h}\right\rangle_{H^{2, 0}}\\
&-\left\langle \partial_y(ABh\partial_x\tilde{h}), \partial_y u\right\rangle_{H^{2, 0}}
-\left\langle \partial_y(ABh\partial_x u), \partial_y\tilde{h}\right\rangle_{H^{2, 0}}\\
&+\left\langle \partial_y\left((1-A)Bh\partial_y\tilde{h}\partial_y u\right), \partial_y u\right\rangle_{H^{2, 0}}
+\left\langle \partial_y\left((1-B)Ah\partial_y\tilde{h}\partial_y \tilde{h}\right), \partial_y\tilde{h}\right\rangle_{H^{2, 0}}\\
&-\left\langle \partial_y(ABh^2\partial_y^2u), \partial_y u\right\rangle_{H^{2, 0}}
-\left\langle \partial_y(ABh^2\partial_y^2\tilde{h}), \partial_y\tilde{h}\right\rangle_{H^{2, 0}}=: \sum\limits_{i=1}^{10}J_i.
\end{align*}
Then, we will estimate $J_i (i=1,...,10)$ term by term.

First, we separate $J_1$ into two parts
\begin{align*}
J_1=\langle B\partial_t\partial_y u,  \partial_y u\rangle_{H^{2, 0}}+\langle \partial_y B\partial_t u,  \partial_y u\rangle_{H^{2, 0}}=: J_1^1+J_1^2.
\end{align*}
For the first term $J_1$, along the same line as (\ref{3.3}), we have
\begin{align*}
J_1^1&=\sum\limits_{|\alpha|=0}^{2}\langle B\partial_\tau^\alpha\partial_t\partial_y u, \partial_\tau^\alpha\partial_y u\rangle+\sum\limits_{|\alpha|=1}^{2}\sum\limits_{|\beta|=1}^{|\alpha|}\left\langle \partial_\tau^\beta B\partial_\tau^{\alpha-\beta}\partial_t\partial_y u, \partial_\tau^\alpha\partial_y u\right\rangle\\
&=\frac 12\frac d{dt}\sum\limits_{|\alpha|=0}^{2}\left\|\sqrt{B}\partial_\tau^\alpha\partial_y u\right\|_{L^2}^2-\frac 12\sum\limits_{|\alpha|=0}^{2}\int_{\Omega} B_t(\partial_\tau^\alpha\partial_y u)^2+\sum\limits_{|\alpha|=1}^{2}\sum\limits_{|\beta|=1}^{|\alpha|}\left\langle \partial_\tau^\beta B\partial_\tau^{\alpha-\beta}\partial_t\partial_y u, \partial_\tau^\alpha\partial_y u\right\rangle\\
&\ge \frac 12\frac d{dt}\sum\limits_{|\alpha|=0}^{2}\left\|\sqrt{B}\partial_\tau^\alpha\partial_y u\right\|_{L^2}^2-C\|\partial_\tau B\|_{L^\infty}\|u\|_{H^{2, 1}}^2-C\|\partial_\tau\partial_y u\|_{L^\infty}\|\partial_\tau B\|_{H^{1, 0}}\|u\|_{H^{2, 1}}\\
&\ge \frac 12\frac d{dt}\sum\limits_{|\alpha|=0}^{2}\left\|\sqrt{B}\partial_\tau^\alpha\partial_y u\right\|_{L^2}^2-CD(t)^{\frac12}E(t),
\end{align*}
where in the third line, we used Lemma \ref{Lem2.1}, and in the last inequality Lemma \ref{Lem2.3} is used. Moreover, by Lemma \ref{Lem2.1} and Gagliardo-Nirenberg inequality, $J_1^2$ can be estimated in the following way.
\begin{align*}
|J_1^2|&\lesssim \|\partial_y B\|_{L^\infty}\|\partial_y u\|_{H^{2, 0}}\|u\|_{H^{2, 1}}+\|\partial_y u\|_{L^\infty}\|\partial_y B\|_{H^{2, 0}}\|u\|_{H^{2, 1}}\\
&\le CD(t)^{\frac12}E(t).
\end{align*}
Combine the above two estimates, it holds
\begin{align}\label{3.14}
J_1\ge  \frac 12\frac d{dt}\sum\limits_{|\alpha|=0}^{2}\left\|\sqrt{B}\partial_\tau^\alpha\partial_y u\right\|_{L^2}^2-CD(t)^{\frac12}E(t).
\end{align}
Similarly,
\begin{align}\label{3.15}
J_2\ge  \frac 12\frac d{dt}\sum\limits_{|\alpha|=0}^{2}\left\|\sqrt{A}\partial_\tau^\alpha\partial_y\tilde{h}\right\|_{L^2}^2-CD(t)^{\frac12}E(t).
\end{align}
For $J_3$, again by Lemma \ref{Lem2.1}, one has
\begin{align}\label{3.16}
|J_3|&\le\left|\left\langle Bu\partial_x\partial_y u, \partial_y u\right\rangle_{H^{2, 0}}\right|+\left|\left\langle \partial_y(Bu)\partial_x u, \partial_y u\right\rangle_{H^{2, 0}}\right|\notag\\
&\lesssim \|Bu\|_{L^\infty}\|\partial_x\partial_y u\|_{H^{2, 0}}\|u\|_{H^{2, 1}}+\|\partial_x\partial_y u\|_{L^\infty}\|Bu\|_{H^{2, 0}}\|u\|_{H^{2, 1}}\notag\\
&\quad +\|\partial_y(Bu)\|_{L^\infty}\|\partial_x u\|_{H^{2, 0}}\|u\|_{H^{2, 1}}+\|\partial_x u\|_{L^\infty}\|\partial_y(Bu)\|_{H^{2, 0}}\|u\|_{H^{2, 1}}\notag\\
&\le CD(t)^{\frac12}E(t).
\end{align}
By the exactly same procedure, $J_4$ has the same bound.
\begin{align}\label{3.17}
|J_4|\le CD(t)^{\frac12}E(t).
\end{align}
Next, we establish the estimate of $J_5+J_6$.
\begin{align*}
&\left\langle \partial_y(ABh\partial_x\tilde{h}), \partial_y u\right\rangle_{H^{2, 0}}
-\left\langle \partial_y(ABh\partial_x u), \partial_y\tilde{h}\right\rangle_{H^{2, 0}}\\
=&\left(\left\langle \partial_y(ABh)\partial_x\tilde{h}, \partial_y u\right\rangle_{H^{2, 0}}
+\left\langle \partial_y(ABh)\partial_x u, \partial_y\tilde{h}\right\rangle_{H^{2, 0}}\right)\\
&\qquad-\left(\left\langle ABh\partial_x\partial_y\tilde{h}, \partial_y u\right\rangle_{H^{2, 0}}
+\left\langle ABh\partial_x\partial_y u, \partial_y\tilde{h}\right\rangle_{H^{2, 0}}\right)\\
=&\left(\left\langle \partial_y(ABh)\partial_x\tilde{h}, \partial_y u\right\rangle_{H^{2, 0}}
+\left\langle \partial_y(ABh)\partial_x u, \partial_y\tilde{h}\right\rangle_{H^{2, 0}}\right)\\
&\qquad+\left\langle \partial_x(ABh)\partial_y\tilde{h}, \partial_y u\right\rangle_{H^{2, 0}}
\end{align*}
where in the last line we use integration by parts. It follows from the same line as (\ref{3.16}), we deduce that
\begin{align}\label{3.18}
|J_5+J_6|\le CD(t)^{\frac12}E(t).
\end{align}
We continue to estimate $J_7$.
\begin{align*}
|J_7|&\le \left|\left\langle \partial_y((1-A)Bh)\partial_y\tilde{h}\partial_y u, \partial_y u\right\rangle_{H^{2, 0}}\right|+ \left|\left\langle (1-A)Bh\partial_y(\partial_y\tilde{h}\partial_y u), \partial_y u\right\rangle_{H^{2, 0}}\right|\notag\\
&\lesssim \left\|\partial_y((1-A)Bh)\right\|_{L^\infty}\|\partial_y\tilde{h}\partial_y u\|_{H^{2, 0}}\|u\|_{H^{2, 1}}+\|\partial_y\tilde{h}\partial_y u\|_{L^\infty}
\left\|\partial_y((1-A)Bh)\right\|_{H^{2, 0}}\|u\|_{H^{2, 1}}\notag\\
&\quad +\|(1-A)Bh\|_{L^\infty}\left\|\partial_y(\partial_y\tilde{h}\partial_y u)\right\|_{H^{2, 0}}\|u\|_{H^{2, 1}}\notag\\
&\quad+\left\|\partial_y(\partial_y\tilde{h}\partial_y u)\right\|_{L_\tau^\infty L_y^2}\left\|(1-A)Bh\right\|_{H_\tau^2 L_y^\infty}\|u\|_{H^{2, 1}}\\
&=: J_7^1+J_7^2+J_7^3+J_7^4.
\end{align*}
Firstly,
\begin{align*}
J_7^1&\le CE(t)^{\frac12}\cdot D(t)E(t)^{\frac12}\le CE(t)D(t).
\end{align*}
For the second and the third terms,
\begin{align*}
J_7^2&\le CD(t)E(t)^{\frac12}E(t)^{\frac12}\le CD(t)E(t),
\end{align*}
and
\begin{align*}
J_7^3\le CD(t)E(t)^{\frac12}.
\end{align*}
As for the last term
\begin{align*}
J_7^4\le CD(t)E(t)^{\frac12}\cdot E(t)^{\frac12}\le CD(t)E(t).
\end{align*}
As a result, by the a priori assumption (\ref{2.6}), it turns out
\begin{align}\label{3.19}
|J_7|\le CD(t)E(t)^{\frac12}.
\end{align}
A similar derivation of the above inequality yields that
\begin{align}\label{3.20}
|J_8|\le CD(t)E(t)^{\frac12}.
\end{align}
It is left to estimate $J_9$, by integration by parts
\begin{align*}
J_9&=\left\langle ABh^2\partial_y^2u, \partial_y^2 u\right\rangle_{H^{2, 0}}-\left\langle ABh^2\partial_y^2u, \partial_y u\right\rangle_{H^{2, 0}}|_{y=0}\\
&=\left\langle ABh^2\partial_\tau^2\partial_y^2u, \partial_\tau^2\partial_y^2 u\right\rangle+\left\langle \partial_\tau(ABh^2)\partial_\tau\partial_y^2u, \partial_\tau^2\partial_y^2 u\right\rangle\\
&\quad +\left\langle \partial_\tau^2(ABh^2)\partial_y^2u, \partial_\tau^2\partial_y^2 u\right\rangle-\left\langle ABh^2\partial_y^2u, \partial_y u\right\rangle_{H^{2, 0}}|_{y=0}\\
&=: J_9^1+J_9^2+J_9^3+J_9^4.
\end{align*}
Obviously
\begin{align*}
J_9^1\ge C^\prime_2\|\partial_y u\|_{H^{2, 1}}^2.
\end{align*}
For $J_9^2$ and $J_9^3$, by Gagliardo-Nirenberg inequality, we have
\begin{align*}
|J_9^2|+|J_9^3|&\lesssim \|\partial_\tau(ABh^2)\|_{L^\infty}\|\partial_y u\|_{H^{1, 1}}\|\partial_y u\|_{H^{2, 1}}+ \|\partial_\tau^2(ABh^2)\|_{L_\tau^2 L_y^\infty}\|\partial_y^2 u\|_{L_\tau^\infty L_y^2}\|\partial_y u\|_{H^{2, 1}}\\
&\le CD(t)E(t)^{\frac12}.
\end{align*}
Notice that the boundary term is the most involved, here we divide it into three parts.
\begin{align*}
&-\langle ABh^2\partial_y^2u, \partial_y u\rangle_{H^{2, 0}}|_{y=0}\\
=&-\langle \partial_\tau(ABh^2)\partial_\tau\partial_y^2u, \partial_\tau^2\partial_y u\rangle|_{y=0}-\langle \partial_\tau^2(ABh^2)\partial_y^2u, \partial_\tau^2\partial_y u\rangle_{H^{2, 0}}|_{y=0}-\langle ABh^2\partial_\tau^2\partial_y^2u, \partial_\tau^2\partial_y u\rangle_{H^{2, 0}}|_{y=0}\\
=:& \quad J_9^{4,1}+J_9^{4,2}+J_9^{4,3}.
\end{align*}
From the first equation in (\ref{2.2}) and the boundary condition (\ref{2.4}), we immediately get that
\[
\partial_y^2 u|_{y=0}=-\frac{1}{h}\partial_x\tilde{h}|_{y=0}.
\]
Substituting it into the first two terms which behave like the nonlinear terms, and by using Lemma \ref{Lem2.1}, we can estimate them directly.
\begin{align*}
&|J_9^{4,1}+J_9^{4,2}|\\
\lesssim& \left\|\partial_\tau(ABh^2)\right\|_{L^\infty}\left\|\partial_{\tau}\left(\frac{1}{h}\partial_x\tilde{h}\right)\right\|_{L_x^2 L_y^\infty}\left\|\partial_\tau^2\partial_y u\right\|_{L_x^2 L_y^\infty}+ \left\|\partial_\tau^2(ABh^2)\right\|_{L^\infty}\left\|\frac{1}{h}\partial_x\tilde{h}\right\|_{L_x^2 L_y^\infty}\left\|\partial_\tau^2\partial_y u\right\|_{L_x^2 L_y^\infty}\\
\le& CD(t)^{\frac12}E(t)+CD(t)E(t)^{\frac12}.
\end{align*}
By Gagliadro-Nirenberg Sobolev embedding inequality again, the last term can be estimated as follows.
\begin{align*}
|J_9^{4,3}|\lesssim \|\partial_\tau^2\partial_y^2 u\|_{L^2}^{\frac12}\|\partial_\tau^2\partial_y^3 u\|_{L^2}^{\frac12}\cdot\|\partial_\tau^2\partial_y u\|_{L^2}^{\frac12}\|\partial_\tau^2\partial_y^2 u\|_{L^2}^{\frac12}.
\end{align*}
The most complicated term comes from the second term in the right hand side of above inequality. To deal with it, we take $\partial_y$ to the first equation in (\ref{2.2}) and obtain
\begin{align*}
\partial_y^3 u=\partial_y\left(\frac{1}{Ah^2}\partial_\tau u+\frac{1}{Ah^2}u\partial_x u-\frac{1}{h}\partial_x\tilde{h}+\frac{1-A}{Ah}\partial_y\tilde{h}\partial_y u\right).
\end{align*}
Hence,
\begin{align*}
\|\partial_\tau^2\partial_y^3 u\|_{L^2}&\le\left\|\partial_y\left(\frac{1}{Ah^2}\partial_\tau u\right)\right\|_{H^{2, 0}}+\left\|\partial_y\left(\frac{1}{Ah^2}u\partial_x u\right)\right\|_{H^{2, 0}}\\
&\qquad+\left\|\partial_y\left(\frac{1}{h}\partial_x\tilde{h}\right)\right\|_{H^{2, 0}}+\left\|\partial_y\left(\frac{1-A}{Ah}\partial_y\tilde{h}\partial_y u\right)\right\|_{H^{2, 0}}\\
&=: j_1+j_2+j_3+j_4.
\end{align*}
All we need to do is to estimate $j_i (j=1,...,4)$. Firstly,
\begin{align*}
|j_1|&\lesssim \left\|\partial_y\left(\frac 1{Ah^2}\right)\partial_\tau u\right\|_{H^{2, 0}}+\left\|\frac 1{Ah^2}\partial_\tau\partial_y u\right\|_{H^{2, 0}}\\
&\lesssim \|\partial_\tau u\|_{L^\infty}\left\|\frac 1{Ah^2}\right\|_{H^{2, 1}}+\left\|\partial_y\left(\frac 1{Ah^2}\right)\right\|_{L^\infty}\|\partial_\tau u\|_{H^{2, 0}}\\
&\qquad\quad+\left\|\frac 1{Ah^2}\right\|_{L^\infty}\|\partial_\tau u\|_{H^{2, 1}}+\left\|\partial_\tau^2\left(\frac 1{Ah^2}\right)\right\|_{L^2}\|\partial_\tau\partial_y u\|_{L^\infty}\\
&\le CD(t)^{\frac12}E(t)^{\frac12}+CD(t)^{\frac12}.
\end{align*}
A similar derivation yields that $j_3$ has the same bound.
\begin{align*}
|j_3|\le CD(t)^{\frac12}E(t)^{\frac12}+CD(t)^{\frac12}.
\end{align*}
Furthermore, by Lemma \ref{Lem2.1} again, there holds
\begin{align*}
|j_2|&\lesssim \left\|\partial_y\left(\frac u{Ah^2}\right)\partial_x u\right\|_{H^{2, 0}}+\left\|\frac u{Ah^2}\partial_x\partial_y u\right\|_{H^{2, 0}}\\
&\lesssim \|\partial_x u\|_{L^\infty}\left\|\frac u{Ah^2}\right\|_{H^{2, 1}}+\left\|\partial_y\left(\frac u{Ah^2}\right)\right\|_{L^\infty}\|\partial_x u\|_{H^{2, 0}}\\
&\qquad\quad+\left\|\frac u{Ah^2}\right\|_{L^\infty}\|\partial_x u\|_{H^{2, 1}}+\left\|\frac u{Ah^2}\right\|_{H^{2, 0}}\|\partial_x\partial_y u\|_{L^\infty}\\
&\le CD(t)^{\frac12}E(t)^{\frac12}.
\end{align*}
By the same trick, $j_4$ can be bounded as follows.
\begin{align*}
|j_4|\le CD(t).
\end{align*}
Collecting these four estimates, it turns out
\begin{align*}
\|\partial_\tau^2\partial_y^3 u\|_{L^2}\le CD(t)^{\frac12}E(t)^{\frac12}+CD(t)+CD(t)^{\frac12}.
\end{align*}
From which we achieve that
\begin{align*}
|J_9^{4,3}|&\lesssim \|u\|_{H^{2, 2}}^{\frac12}\|u\|_{H^{2, 1}}^{\frac12}\|u\|_{H^{2, 2}}^{\frac12}\cdot\left(D(t)^{\frac14}+D(t)^{\frac14}E(t)^{\frac14}+D(t)^{\frac12}\right)\\
&\le \frac{C}{M^{\frac14}}D(t)+CD(t)E(t)^{\frac14}.
\end{align*}
Thus, we infer that
\begin{align}\label{3.21}
J_9\ge C_2\|\partial_y u\|_{H^{2, 1}}^2-CD(t)^{\frac12}E(t)-CD(t)E(t)^{\frac14}-\frac{C}{M^{\frac14}} D(t).
\end{align}
Finally, since there is no boundary term when applying integration by parts to $J_{10}$, the estimate of $J_{10}$ is much more concise.
\begin{align}\label{3.22}
J_{10}\ge C_2\|\partial_y\tilde{h}\|_{H^{2, 1}}^2-CD(t)E(t)^{\frac12}.
\end{align}
Consequently, summing up all the estimates (\ref{3.14})-(\ref{3.22}) together, we arrive at
\begin{align}\label{3.23}
\frac d{dt}\sum\limits_{|\alpha|=0}^{2}(\|\sqrt{B}\partial_\tau^\alpha\partial_y u\|_{L^2}^2&+\|\sqrt{A}\partial_\tau^\alpha\partial_y\tilde{h}\|_{L^2}^2)+C_2\|\partial_y(u, \tilde{h})\|_{H^{2, 1}}^2\notag\\
&\le CD(t)^{\frac 12}E(t)+CD(t)E(t)^{\frac14}+\frac{C}{M^{\frac14}} D(t).
\end{align}
Hence, we complete the proof of Proposition \ref{Prop3.3}.
\end{proof}

Now it is position to start the proof of Proposition \ref{Prop3.1}. Based on the Proposition \ref{Prop3.2} and \ref{Prop3.3}, we conclude that for any $\alpha\in\mathbb{N}^2, \beta\in\mathbb{N}$, if we take $M$ large enough, then we set up the following inequality.
\begin{align}\label{3.24}
\frac d{dt}\sum\limits_{\beta=0}^{1}\sum\limits_{|\alpha|=0}^{3-\beta}(\|\sqrt{B}\partial_\tau^\alpha\partial_y^\beta u\|_{L^2}^2&+\|\sqrt{A}\partial_\tau^\alpha\partial_y^\beta\tilde{h}\|_{L^2}^2)+C_0D(t) \notag\\
&\le CD(t)^{\frac 14}E(t)^{\frac54}+CD(t)^{\frac 12}E(t)+CD(t)E(t)^{\frac14}.
\end{align}
Then, the proof of Proposition \ref{Prop3.1} is complete.

\subsection{Lower bound estimate of lifespan of solutions}
This subsection is devoted to proving Theorem \ref{Thm1}. Recalling the initial condition (\ref{1.17}) and the coordinate transformation, we immediately get that
\begin{align}\label{3.25}
E(0)\le 2\varepsilon^2.
\end{align}
Furthermore, in view of the basic energy estimates achieved in Subsections 3.1 and 3.2, by H\"older's inequality and using the a priori assumption (\ref{2.6}), the smallness of $\varepsilon$, we have
\begin{align}\label{3.26}
\frac d{dt}\sum\limits_{\beta=0}^{1}\sum\limits_{|\alpha|=0}^{3-\beta}\left(\|\sqrt{B}\partial_\tau^\alpha\partial_y^\beta u\|_{L^2}^2+\|\sqrt{A}\partial_\tau^\alpha\partial_y^\beta\tilde{h}\|_{L^2}^2\right)+\frac {C_0}{2}D(t)\le CE(t)^{\frac53}.
\end{align}
Suppose that $(0, T_\varepsilon)$ is the maximum interval that (\ref{2.6}) holds, then as a consequence of Gronwall's inequality, we deduce that for any $t\in (0, T_\varepsilon)$ with $T_\varepsilon=\frac{\ln 2}{4C}\varepsilon^{-\frac43}$
\begin{align*}
E(t)\le E(0)\exp\left\{\int_0^{T_\varepsilon} CE(t)^{\frac23} dt\right\}\le 2\exp\{4C\varepsilon^{\frac43}T_\varepsilon\}\varepsilon^2= 4\varepsilon^2.
\end{align*}
Lastly, Theorem \ref{Thm1} follows by a bootstrap argument.

\section{Non-isentropic MHD boundary layer equations}
In this section, we consider the positive lower bound of lifespan of solutions to the initial boundary value problem of non-isentropic compressible MHD boundary layer equations (\ref{1.1})-(\ref{1.4}).

\subsection{Coordinate transformation}
To prove Theorem \ref{Thm2}, we also introduce the following coordinate transform
\[
\bar{t}=t,\quad \bar{x}=x,\quad \bar{y}=\psi(t, x, y),
\]
and the new variables as before
\[
(\hat{u}, \hat{\theta}, \hat{h})(\bar{t}, \bar{x}, \bar{y}):=(u, \theta, h)(t, x, y).
\]
Then (\ref{4.1}) becomes
\begin{align}\label{4.4}
\begin{cases}
\partial_{\bar{t}}u+u\partial_{\bar{x}}u-\frac{R\theta}{P-\frac12h^2}h\partial_{\bar{x}}h+\left(1-\frac{R\theta}{P-\frac12h^2}\right)h\partial_{\bar{y}}h\partial_{\bar{y}}u-\frac{R\theta}{P-\frac12h^2}h^2\partial_{\bar{y}}^2u=-\frac{R\theta}{P-\frac12h^2}P_{\bar{x}},\\
\partial_{\bar{t}}\theta+u\partial_{\bar{x}}\theta+\frac{a\theta}{Q}h^2\partial_{\bar{x}}u-\frac{a\theta}{Q}\frac{P+\frac12h^2}{P-\frac12h^2}h^2(\partial_{\bar{y}}u)^2\\
\qquad+\left(1-\frac{a\theta}{Q}\frac{P+\frac12h^2}{P-\frac12h^2}\right)h\partial_{\bar{y}}h\partial_{\bar{y}}\theta-\frac{a\theta}{Q}\frac{P+\frac12h^2}{P-\frac12h^2}h^2(\partial_{\bar{y}}h)^2\\
\qquad-\frac{a\theta}{Q}h^2\left[\frac{P+\frac12h^2}{P-\frac12h^2}\partial_{\bar{y}}^2\theta-\partial_{\bar{y}}(h\partial_{\bar{y}}h)\right]=\frac{a\theta}{Q}(P_{\bar{t}}+P_{\bar{x}}u),\\
\partial_{\bar{t}}h+u\partial_{\bar{x}}h-\frac{P-\frac12h^2}{Q}h\partial_{\bar{x}}u+\frac{a}{Q}h^3(\partial_{\bar{y}}u)^2+\frac{a}{Q}h^2\partial_{\bar{y}}h\partial_{\bar{y}}\theta+\frac{P+\frac12h^2}{Q}h(\partial_{\bar{y}}h)^2\\
\qquad-\frac{P-\frac12h^2}{Q}h\left[\partial_{\bar{y}}(h\partial_{\bar{y}}h)-\frac{ah^2}{P-\frac12h^2}\partial_{\bar{y}}^2\theta\right]=\frac{h}{Q(R+1)}(P_{\bar{t}}+P_{\bar{x}}u).
\end{cases}
\end{align}
Here, we also drop all ``hat" for convenience. As in \cite{HLY}, set
\[
q(\bar{t}, \bar{x}, \bar{y}):=\frac12h^2(\bar{t}, \bar{x}, \bar{y}),
\]
the system (\ref{4.4}) can be rewritten in the following form
\begin{align}\label{4.5}
\begin{cases}
\partial_tu+u\partial_xu-\frac{R\theta}{P-q}\partial_xq+\left(1-\frac{R\theta}{P-q}\right)\partial_yq\partial_yu-\frac{2R\theta q}{P-q}\partial_y^2u=-\frac{R\theta}{P-\frac12h^2}P_x\\
\partial_t\theta+u\partial_x\theta+\frac{2a\theta q}{Q}\partial_xu-\frac{2a\theta q}{Q}\frac{P+q}{P-q}(\partial_yu)^2+\left(1-\frac{a\theta}{Q}\frac{P+q}{P-q}\right)\partial_yq\partial_y\theta\\
\qquad\qquad-\frac{a\theta}{Q}\frac{P+q}{P-q}(\partial_yq)^2-\frac{2a\theta q}{Q}\left[\frac{P+q}{P-q}\partial_y^2\theta-\partial_y^2q\right]=\frac{a\theta}{Q}(P_t+P_xu),\\
\partial_tq+u\partial_xq-\frac{2(P-q)q}{Q}\partial_xu+\frac{4aq^2}{Q}(\partial_yu)^2+\frac{2aq}{Q}\partial_yq\partial_y\theta+\frac{P+q}{Q}(\partial_yq)^2\\
\qquad\qquad-\frac{2(P-q)q}{Q}\left[\partial_y^2q-\frac{2aq}{P-q}\partial_y^2\theta\right]=\frac{2q}{Q(R+1)}(P_t+P_xu),\\
\lim\limits_{y\to\infty}(u, \theta, q)= (0, \Theta, H^2/2)(t, x),\\
(u, \partial_yq)|_{y=0}={\bf{0}}, \qquad\theta|_{y=0}=\theta^*(t, x).
\end{cases}
\end{align}
Here, we also replace $(\bar{t}, \bar{x}, \bar{y})$ with $(t, x, y)$ without any confusion. To overcome the difficulty originated from the boundary term, we introduce a cut-off function $\chi(y)\in C^\infty(\mathbb{R}_+)$ such that $0\le\chi(y)\le 1$,
\begin{align*}
\chi(y)=
\begin{cases}
0, \quad y\in[0, 1],\\
1, \quad y\ge 2,
\end{cases}
\end{align*}
and denote
\[
\tilde{u}=u,\quad \tilde{\theta}=\theta-\chi(y)\Theta-(1-\chi(y))\theta^*, \quad \tilde{q}=q-H^2/2.
\]
Then $(\tilde{u}, \tilde{\theta}, \tilde{q})$ solves the following system of equations.
\begin{align}\label{4.6}
\begin{cases}
\partial_t\tilde{u}+u\partial_x\tilde{u}-\frac{R\theta}{P-q}\partial_x\tilde{q}+\left(1-\frac{R\theta}{P-q}\right)\partial_y\tilde{q}\partial_y\tilde{u}-\frac{2R\theta q}{P-q}\partial_y^2\tilde{u}=r_1,\\
\partial_t\tilde{\theta}+u\partial_x\tilde{\theta}+\frac{2a\theta q}{Q}\partial_x\tilde{u}-\frac{2a\theta q}{Q}\frac{P+q}{P-q}(\partial_y\tilde{u})^2+\left(1-\frac{a\theta}{Q}\frac{P+q}{P-q}\right)\partial_y\tilde{q}\partial_y\tilde{\theta}\\
\qquad\qquad-\frac{a\theta}{Q}\frac{P+q}{P-q}(\partial_y\tilde{q})^2-\frac{2a\theta q}{Q}\left[\frac{P+q}{P-q}\partial_y^2\tilde{\theta}-\partial_y^2\tilde{q}\right]=r_2,\\
\partial_t\tilde{q}+u\partial_x\tilde{q}-\frac{2(P-q)q}{Q}\partial_x\tilde{u}+\frac{4aq^2}{Q}(\partial_y\tilde{u})^2+\frac{2aq}{Q}\partial_y\tilde{q}\partial_y\tilde{\theta}+\frac{P+q}{Q}(\partial_y\tilde{q})^2\\
\qquad\qquad-\frac{2(P-q)q}{Q}\left[\partial_y^2\tilde{q}-\frac{2aq}{P-q}\partial_y^2\tilde{\theta}\right]=r_3.
\end{cases}
\end{align}
By the Bernoulli's law of the outflow, we have
\begin{align*}
\begin{cases}
r_1=-\frac{R\theta}{P-q}P_{x}+\frac{R\theta}{P-q}HH_x=0,\\
r_2=-\chi\Theta_t-(1-\chi)\theta^*_t-\tilde{u}\chi\Theta_x-\tilde{u}(1-\chi)\theta^*_x-\left(1-\frac{a\theta}{Q}\frac{P+q}{P-q}\right)\partial_y\tilde{q}(\chi'\Theta-\chi'\theta^*)\\
\qquad\qquad\qquad+\frac{2a\theta q}{Q}\frac{P+q}{P-q}(\chi''\Theta-\chi''\theta^*)+\frac{a\theta}{Q}(P_t+P_x\tilde{u}),\\
\quad=-(1-\chi)\theta^*_t-\tilde{u}\chi\Theta_x-\tilde{u}(1-\chi)\theta^*_x-\left(1-\frac{a\theta}{Q}\frac{P+q}{P-q}\right)\partial_y\tilde{q}(\chi'\Theta-\chi'\theta^*)\\
\qquad\qquad\qquad+\frac{2a\theta q}{Q}\frac{P+q}{P-q}(\chi''\Theta-\chi''\theta^*)+\frac{a\theta}{Q}\left((1-\chi)P_t+P_x\tilde{u}\right)-\frac{a\Theta P_t\chi\tilde{q}}{\left(P+\frac12(1-2a)H^2\right)Q},\\
r_3=-HH_t-\tilde{u}HH_x-\frac{2aq}{Q}\partial_y\tilde{q}(\chi'\Theta-\chi'\theta^*)-\frac{4aq^2}{Q}(\chi''\Theta-\chi''\theta^*)+\frac{2q}{Q(R+1)}(P_t+P_x\tilde{u})\\
\quad=-\tilde{u}HH_x-\frac{2aq}{Q}\partial_y\tilde{q}(\chi'\Theta-\chi'\theta^*)-\frac{4aq^2}{Q}(\chi''\Theta-\chi''\theta^*)+\frac{2\tilde{q}PP_t}{Q(R+1)(P+\frac12(1-2a)H^2)}+\frac{2q}{Q(R+1)}P_x\tilde{u}.
\end{cases}
\end{align*}
Moreover, by the assumption (\ref{4.2}), it holds
\begin{align}\label{4.7}
\|(r_2, r_3)\|_{H^{3, 0}}+\|(r_2, r_3)\|_{H^{2, 1}}\le Cf(t)+Cf(t)E(t)^{\frac12}+Cf(t)D(t)^{\frac12}.
\end{align}
Let ${\bf{v}}={\bf{v}}(t, x, y):=(\tilde{u}, \tilde{\theta}, \tilde{q})^T(t, x, y)$, the system (\ref{4.6}) can be rewritten in the following form
\begin{align}\label{4.8}
\partial_t{\bf{v}}+{\bf{A_0(v)}}\partial_x{\bf{v}}+f_0(\partial_y{\bf{v}})-{\bf{B_0(v}})\partial_y^2{\bf{v}}=g_0({\bf{v}}),
\end{align}
where
\begin{align*}
{\bf{A_0(v)}}=\left(
\begin{array}{ccc}
u & 0 & -\frac{R\theta}{a-q}\\
\frac{2a\theta q}{Q} & u & 0\\
-\frac{2(P-q)q}{Q} & 0 & u
\end{array}\right),\\
{\bf{B_0(v)}}=2q\left(
\begin{array}{ccc}
-\frac{R\theta}{P-q} & 0 & 0\\
0 & -\frac{a\theta}{Q}\frac{P+q}{P-q} & \frac{a\theta}{Q}\\
0 & \frac{2aq}{Q} & -\frac{P-q}{Q}
\end{array}\right),
\end{align*}
and
\begin{align*}
f_0(\partial_y{\bf{v}})=\left(
\begin{array}{c}
\left(1-\frac{R\theta}{P-q}\right)\partial_{y}\tilde{q}\partial_y\tilde{u}\\
-\frac{2a\theta q}{Q}\frac{P+q}{P-q}(\partial_{y}\tilde{u})^2+\left(1-\frac{a\theta}{Q}\frac{P+q}{P-q}\right)\partial_y\tilde{q}\partial_{y}\tilde{\theta}-\frac{a\theta}{Q}\frac{P+q}{P-q}(\partial_{y}\tilde{q})^2\\
\frac{4aq^2}{Q}(\partial_{y}\tilde{u})^2+\frac{2aq}{Q}h^2\partial_{y}\tilde{q}\partial_{y}\tilde{\theta}+\frac{P+q}{Q}h(\partial_{y}\tilde{q})^2
\end{array}
\right)=:\left(f_1, f_2, f_3\right)^T,
\end{align*}
\begin{align*}
g_0({\bf{v}})=\left(0, r_2, r_3\right)^T.
\end{align*}
We continue to introduce a positive symmetric matrix
\begin{align}\label{4.9}
S({\bf{v}}):=\left(
\begin{array}{ccc}
\frac{\theta(P-q)}{R} & 0 & 0\\
0 & \frac{P-q}{a} & \theta\\
0 & \theta & \frac{\theta^2}{2q}\frac{P+q}{P-q}
\end{array}
\right)
\end{align}
to symmetrize the system. A direct calculation yields that
\begin{align}\label{4.10}
A({\bf{v}}):=S({\bf{v}})A_0({\bf{v}})=\left(
\begin{array}{ccc}
\frac{\theta(P-q)}{R}u & 0 & -\theta^2\\
0 & \frac{P-q}{a}u & \theta u\\
-\theta^2 & \theta u & \frac{\theta^2}{2q}\frac{P+q}{P-q}u
\end{array}
\right)
\end{align}
is symmetric, and
\begin{align}\label{4.11}
B({\bf{v}}):=S({\bf{v}})B_0({\bf{v}})=\left(
\begin{array}{ccc}
2\theta^2q & 0 & 0\\
0 & 2\theta q & 0\\
0 & 0 & \theta^2
\end{array}
\right)
\end{align}
is positive definite. Finally, the system (\ref{1.1}) is converted to
\begin{align}\label{4.12}
\begin{cases}
S({\bf{v}})\partial_t{\bf{v}}+{\bf{A(v)}}\partial_x{\bf{v}}+f(\partial_y{\bf{v}})-{\bf{B(v}})\partial_y^2{\bf{v}}=g({\bf{v}}),\\
(\tilde{u}, \partial_y\tilde{q})|_{y=0}={\bf{0}},\qquad \tilde{\theta}|_{y=0}=0,\\
\lim\limits_{y\to\infty}{\bf{v}}(t, x, y)={\bf{0}},\\
{\bf{v}}|_{t=0}=(\tilde{u}_0, \tilde{\theta}_0, (\tilde{h}_0)^2/2)^T(x, y),
\end{cases}
\end{align}
where
\begin{align}\label{4.13}
f(\partial_y{\bf{v}})=\left(\frac{\theta(P-q)}{R}f_1, \frac{P-q}{a}f_2+\theta f_3, \theta f_2+\frac{\theta^2}{2q}\frac{P+q}{P-q}f_3\right)^T,
\end{align}
and
\begin{align}\label{4.14}
g({\bf{v}})=\left(0, \frac{P-q}{a}r_2+\theta r_3, \theta r_2+\frac{\theta^2}{2q}\frac{P+q}{P-q}r_3\right)^T.
\end{align}
To obtain the lifespan of the solution to the initial boundary value problem (\ref{4.12}), we define the following new energy functionals
\begin{align}\label{4.15}
&E(t)=M\sum\limits_{|\alpha|=0}^{3}\left(\|\sqrt{G_1}\partial_\tau^\alpha\tilde{u}\|_{L^2}^2+\|\sqrt{G_2}\partial_\tau^\alpha\tilde{\theta}\|_{L^2}^2+\|\sqrt{G_3}\partial_\tau^\alpha\tilde{q}\|_{L^2}^2\right)\notag\\
&\qquad\quad+\sum\limits_{|\alpha|=0}^{2}\left(\|\sqrt{G_1}\partial_\tau^\alpha\partial_y\tilde{u}\|_{L^2}^2+\|\sqrt{G_2}\partial_\tau^\alpha\partial_y\tilde{\theta}\|_{L^2}^2+\|\sqrt{G_3}\partial_\tau^\alpha\partial_y\tilde{q}\|_{L^2}^2\right),\notag\\
&D(t)=M\|\partial_{y}(\tilde{u}, \tilde{\theta}, \tilde{q})\|_{H^{3, 0}}^2+\|\partial_{y}(\tilde{u}, \tilde{\theta}, \tilde{q})\|_{H^{2, 1}},
\end{align}
where
\begin{align}\label{G}
G_1=\frac{\theta(P-q)}{R},\quad G_2=\frac{P-q}{a}, \quad G_3=\frac{\theta^2}{2q}\frac{P+q}{P-q},
\end{align}
and $M$ is a large constant which will be determined later. From now on, we always assume $(\tilde{u}, \tilde{\theta}, \tilde{h})$ is a smooth solution to (\ref{4.12}), and there exists a time $T$ such that the following a priori assumption
\begin{align}\label{4.16}
E(t)\le 8\varepsilon^2
\end{align}
holds for any $t\in[0, T]$. The proof of Theorem \ref{Thm2} relies on the following proposition.
\begin{proposition}\label{Prop4.3}
Under the a priori assumption (\ref{4.16}), for ant $t\in (0, T)$, there exist constants $C, C_0>0$, such that for any $\alpha\in\mathbb{N}^2, \beta\in\mathbb{N}$, it holds
\begin{align}\label{4.17}
\frac d{dt}&\sum\limits_{\beta=0}^{1}\sum\limits_{|\alpha|=0}^{3-\beta}\left(\|\sqrt{G_1}\partial_\tau^\alpha\partial_y^\beta\tilde{u}\|_{L^2}^2+\|\sqrt{G_2}\partial_\tau^\alpha\partial_y^\beta\tilde{\theta}\|_{L^2}^2+\|\sqrt{G_3}\partial_\tau^\alpha\partial_y^\beta\tilde{q}\|_{L^2}^2\right)+C_0D(t) \notag\\
&\le CD(t)^{\frac 14}E(t)^{\frac54}+CD(t)^{\frac 12}E(t)+CD(t)E(t)^{\frac14}+Cf(t)E(t)^{\frac12}+Cf(t)D(t)^{\frac12}E(t)^{\frac12}.
\end{align}
\end{proposition}

\subsection{Tangential derivative estimates}
The target of this subsection is to establish $H^{3, 0}$ estimates of the solution $(\tilde{u}, \tilde{\theta}, \tilde{h})$ to the initial boundary value problem (\ref{4.12}).
\begin{proposition}\label{Prop4.4}
Under the a priori assumption (\ref{4.16}), for any $t\in (0, T)$, there exist constants $C, C_1>0$, such that
\begin{align}\label{4.18}
\frac d{dt}&\sum\limits_{|\alpha|=0}^{3}\left(\|\sqrt{G_1}\partial_\tau^\alpha\tilde{u}\|_{L^2}^2+\|\sqrt{G_2}\partial_\tau^\alpha\tilde{\theta}\|_{L^2}^2+\|\sqrt{G_3}\partial_\tau^\alpha\tilde{q}\|_{L^2}^2\right)+C_1\|\partial_y(\tilde{u}, \tilde{\theta}, \tilde{q})\|_{H^{3, 0}}^2\notag\\
&\le CD(t)^{\frac14}E(t)^{\frac54}+CD(t)^{\frac 12}E(t)+CD(t)E(t)^{\frac12}+Cf(t)E(t)^{\frac12}+Cf(t)D(t)^{\frac12}E(t)^{\frac12}.
\end{align}
\end{proposition}

\begin{proof}
Taking the $H^{3, 0}$ inner product on the first equation in (\ref{4.12}) with ${\bf{v}}$, we have
\begin{align*}
0=\langle S({\bf{v}})\partial_t{\bf{v}}&,  {\bf{v}}\rangle_{H^{3, 0}}+\langle A({\bf{v}})\partial_x{\bf{v}}, {\bf{v}}\rangle_{H^{3, 0}}
+\left\langle f({\bf{v}}, \partial_y{\bf{v}}), {\bf{v}}\right\rangle_{H^{3, 0}}\\
&-\left\langle B({\bf{v}})\partial_y^2{\bf{v}}, {\bf{v}}\right\rangle_{H^{3, 0}}-\left\langle g({\bf{v}}), {\bf{v}}\right\rangle_{H^{3, 0}}=: \sum\limits_{i=1}^{5}K_i.
\end{align*}
Next, we will estimate $K_i (i=1,...,5)$ term by term.
\begin{align}\label{4.19}
K_1&=\sum\limits_{|\alpha|=0}^{3}\left\langle\partial_\tau^\alpha(S({\bf{v}})\partial_t{\bf{v}}), \partial_\tau^\alpha {\bf{v}}\right\rangle=\sum\limits_{|\alpha|=0}^{3}\left\langle S({\bf{v}})\partial_\tau^\alpha\partial_t {\bf{v}}, \partial_\tau^\alpha {\bf{v}}\right\rangle
+\sum\limits_{|\alpha|=1}^{3}\sum\limits_{|\beta|=1}^{|\alpha|}\left\langle \partial_\tau^\beta S({\bf{v}})\partial_\tau^{\alpha-\beta}\partial_t {\bf{v}}, \partial_\tau^\alpha {\bf{v}}\right\rangle\notag\\
&=\frac 12\frac d{dt}\sum\limits_{|\alpha|=0}^{3}\left\|\sqrt{S({\bf{v}})}\partial_\tau^\alpha {\bf{v}}\right\|_{L^2}^2-\frac 12\sum\limits_{|\alpha|=0}^{3}\int_{\Omega} S({\bf{v}})_t(\partial_\tau^\alpha {\bf{v}})^2+\sum\limits_{|\alpha|=1}^{3}\sum\limits_{|\beta|=1}^{|\alpha|}\left\langle \partial_\tau^\beta S({\bf{v}})\partial_\tau^{\alpha-\beta}\partial_t {\bf{v}},
\partial_\tau^\alpha {\bf{v}}\right\rangle\notag\\
&\ge \frac 12\frac d{dt}\sum\limits_{|\alpha|=0}^{3}\left\|\sqrt{S({\bf{v}})}\partial_\tau^\alpha {\bf{v}}\right\|_{L^2}^2-C\|\partial_\tau S({\bf{v}})\|_{L^\infty}\|{\bf{v}}\|_{H^{3, 0}}^2-C\|\partial_\tau {\bf{v}}\|_{L_x^\infty L_y^2}\|\partial_\tau S({\bf{v}})\|_{H_x^2L_y^\infty}\|{\bf{v}}\|_{H^{3, 0}}\notag\\
&\ge \frac 12\frac d{dt}\sum\limits_{|\alpha|=0}^{3}\left\|\sqrt{S({\bf{v}})}\partial_\tau^\alpha {\bf{v}}\right\|_{L^2}^2-CD(t)^{\frac14}E(t)^{\frac54}-Cf(t)E(t),
\end{align}
where in the last inequality, (\ref{2.7}) is used. Next, we estimate $K_2$.
\begin{align}\label{4.20}
|K_2|&\le\sum\limits_{|\alpha|=0}^{3}\left|\left\langle\partial_\tau^\alpha(A({\bf{v}})\partial_x{\bf{v}}), \partial_\tau^\alpha {\bf{v}}\right\rangle\right |\notag\\
&\le\sum\limits_{|\alpha|=0}^{3}\left|\langle A({\bf{v}})\partial_x\partial_\tau^\alpha {\bf{v}}, \partial_\tau^\alpha {\bf{v}}\rangle\right|
+\sum\limits_{|\alpha|=0}^{3}\sum\limits_{|\beta|=1}^{|\alpha|}\left|\left\langle\partial_\tau^\beta A({\bf{v}})\partial_\tau^{\alpha+1-\beta} {\bf{v}}, \partial_\tau^\alpha {\bf{v}}\right\rangle_{L^2}\right|\notag\\
&=\sum\limits_{|\alpha|=0}^{3}\left|-\frac 12\int_{\Omega} \partial_x A({\bf{v}})(\partial_\tau^\alpha {\bf{v}})^2\right|
+\sum\limits_{|\alpha|=0}^{3}\sum\limits_{|\beta|=1}^{|\alpha|}\left|\left\langle\partial_\tau^\beta A({\bf{v}})\partial_\tau^{\alpha+1-\beta} {\bf{v}}, \partial_\tau^\alpha {\bf{v}}\right\rangle_{L^2}\right|\notag\\
&\lesssim\|\partial_\tau A({\bf{v}})\|_{L^\infty}\|{\bf{v}}\|_{H^{3, 0}}^2+\|\partial_\tau {\bf{v}}\|_{L_x^\infty L_y^2}\|\partial_\tau A({\bf{v}})\|_{H_x^2L_y^\infty}\|{\bf{v}}\|_{H^{3, 0}}\notag\\
&\le CD(t)^{\frac14}E(t)^{\frac54}-Cf(t)E(t).
\end{align}
We move to estimate $K_3$. For the first component of $f({\bf{v}}, \partial_y{\bf{v}})$, it can be estimated directly.
\begin{align*}
&\left|\left\langle \frac{(P-q-R\theta)\theta}{R}\partial_y\tilde{q}\partial_y\tilde{u},\tilde{u}\right\rangle_{H^{3, 0}}\right|\\
&\lesssim\left\|(P-q-R\theta)\theta)\right\|_{L^\infty}\|\partial_y\tilde{q}\partial_y\tilde{u}\|_{H^{3, 0}}\|\tilde{u}\|_{H^{3, 0}}+\|\partial_y\tilde{q}\partial_y\tilde{u}\|_{L^\infty}\|(P-q-R\theta)\theta\|_{H^{3, 0}}\|\tilde{u}\|_{H^{3, 0}}.
\end{align*}
By the same argument as (\ref{3.8}), we arrive at
\begin{align*}
\left|\left\langle \frac{(P-q-R\theta)\theta}{R}\partial_y\tilde{q}\partial_y\tilde{u},\tilde{u}\right\rangle_{H^{3, 0}}\right|\le CD(t)E(t)+CD(t)E(t)^{\frac12}\le CD(t)E(t)^{\frac12}.
\end{align*}
Notice that the other components of $f({\bf{v}}, \partial_y{\bf{v}})$ can be bounded as above. Thus, we obtain the estimate of $K_3$,
\begin{align}\label{4.21}
|K_3|\le CD(t)E(t)^{\frac12}.
\end{align}
As for $K_4$, we separate it into two parts
\begin{align*}
K_4&=-\sum\limits_{|\alpha|=0}^{3}\left\langle B({\bf{v}})\partial_y^2\partial_\tau^\alpha {\bf{v}}, \partial_\tau^\alpha {\bf{v}}\right\rangle-\sum\limits_{|\alpha|=0}^{3}\sum\limits_{|\beta|=1}^{|\alpha|}\left\langle\partial_\tau^\beta B({\bf{v}})\partial_\tau^{\alpha-\beta}\partial_y^2 {\bf{v}}, \partial_\tau^\alpha {\bf{v}}\right\rangle\\
&=:K_4^1+K_4^2.
\end{align*}
By integration by parts, the first one can be estimated as follows.
\begin{align*}
K_4^1&=\sum\limits_{|\alpha|=0}^{3}\int_{\Omega}B({\bf{v}})(\partial_\tau^\alpha\partial_y {\bf{v}})^2+\sum\limits_{|\alpha|=0}^{3}\left\langle\partial_y B({\bf{v}})\partial_\tau^\alpha\partial_y {\bf{v}}, \partial_\tau^\alpha {\bf{v}}\right\rangle\\
&\ge C\|\partial_y {\bf{v}}\|_{H^{3, 0}}^2-\|\partial_y B({\bf{v}})\|_{L^\infty}\|\partial_y {\bf{v}}\|_{H^{3, 0}}\|{\bf{v}}\|_{H^{3, 0}}\\
&\ge C\|\partial_y {\bf{v}}\|_{H^{3, 0}}^2-CD(t)E(t)^{\frac12}-Cf(t)D(t)^{\frac12}E(t)^{\frac12}.
\end{align*}
And for the second part, by Lemma \ref{Lem2.1}, we have
\begin{align*}
|K_4^2|&\lesssim \|\partial_\tau B({\bf{v}})\|_{L^\infty}\|\partial_y^2 {\bf{v}}\|_{H^{2, 0}}\|{\bf{v}}\|_{H^{3, 0}}+\|\partial_y^2 {\bf{v}}\|_{L_x^\infty L_y^2}\|\partial_\tau^3 B({\bf{v}})\|_{L_x^2L_y^\infty}\|{\bf{v}}\|_{H^{3, 0}}\\
&\le CD(t)^{\frac 12}E(t)+CD(t)E(t)^{\frac 12}+Cf(t)D(t)^{\frac12}E(t)^{\frac12}.
\end{align*}
Consequently,
\begin{align}\label{4.22}
K_4\ge C_1\|\partial_y {\bf{v}}\|_{H^{3, 0}}^2- CD(t)^{\frac 12}E(t)-CD(t)E(t)^{\frac12}-Cf(t)D(t)^{\frac12}E(t)^{\frac12}.
\end{align}
Finally, by (\ref{4.7}), it is direct to estimate the last term as follows.
\begin{align}\label{4.23}
|K_5|\le Cf(t)E(t)^{\frac12}+Cf(t)E(t)+Cf(t)D(t)^{\frac12}E(t)^{\frac12}.
\end{align}
Finally, collecting all the estimates (\ref{4.19})-(\ref{4.23}) together, we conclude that
\begin{align}\label{4.24}
\frac d{dt}&\sum\limits_{|\alpha|=0}^{3}\left(\|\sqrt{G_1}\partial_\tau^\alpha\tilde{u}\|_{L^2}^2+\|\sqrt{G_2}\partial_\tau^\alpha\tilde{\theta}\|_{L^2}^2+\|\sqrt{G_3}\partial_\tau^\alpha\tilde{q}\|_{L^2}^2\right)+C_1\|\partial_y(\tilde{u}, \tilde{\theta}, \tilde{q})\|_{H^{3, 0}}^2\notag\\
&\le CD(t)^{\frac14}E(t)^{\frac54}+CD(t)^{\frac 12}E(t)+CD(t)E(t)^{\frac12}+Cf(t)E(t)^{\frac12}+Cf(t)D(t)^{\frac12}E(t)^{\frac12}.
\end{align}
And we finish the proof of Proposition \ref{Prop4.4}.
\end{proof}

\subsection{Normal derivative estimates}
The aim of this subsection is to establish $H^{2, 1}$ estimates of the solution $(\tilde{u}, \tilde{\theta}, \tilde{h})$ to the initial boundary value problem (\ref{4.12}).
\begin{proposition}\label{Prop4.5}
Under the a priori assumption (\ref{4.16}), for any $t\in (0, T)$, there exist constants $C, C_2>0$, such that for suitably large $M$, it holds
\begin{align}\label{4.25}
&\frac d{dt}\sum\limits_{|\alpha|=0}^{2}\left(\|\sqrt{G_1}\partial_\tau^\alpha\partial_y\tilde{u}\|_{L^2}^2+\|\sqrt{G_2}\partial_\tau^\alpha\partial_y\tilde{\theta}\|_{L^2}^2+\|\sqrt{G_3}\partial_\tau^\alpha\partial_y\tilde{q}\|_{L^2}^2\right)+C_1\|\partial_y(\tilde{u}, \tilde{\theta}, \tilde{q})\|_{H^{2, 1}}^2\notag\\
&\le CD(t)E(t)^{\frac14}+CD(t)^{\frac 12}E(t)+\frac{C}{M^{\frac14}}D(t)+Cf(t)E(t)^{\frac12}+Cf(t)D(t)^{\frac12}E(t)^{\frac12}.
\end{align}
\end{proposition}

\begin{proof}
Applying $\partial_y$ on the first equation of (\ref{4.12}) and take the $H^{2, 0}$ inner product on the resulting equation with $\partial_y{\bf{v}}$, we have
\begin{align*}
0=\langle \partial_y\left(S({\bf{v}})\partial_t{\bf{v}}\right)&, \partial_y{\bf{v}}\rangle_{H^{2, 0}}+\left\langle \partial_y\left(A({\bf{v}})\partial_x{\bf{v}}\right), \partial_y{\bf{v}}\right\rangle_{H^{2, 0}}
+\left\langle \partial_yf({\bf{v}}, \partial_y{\bf{v}}), \partial_y{\bf{v}}\right\rangle_{H^{2, 0}}\\
&-\left\langle \partial_y\left(B({\bf{v}})\partial_y^2{\bf{v}}\right), \partial_y{\bf{v}}\right\rangle_{H^{2, 0}}-\left\langle \partial_yg({\bf{v}}), \partial_y{\bf{v}}\right\rangle_{H^{2, 0}}=: \sum\limits_{i=1}^{5}L_i.
\end{align*}
Now we handle $L_i (i=1,...,5)$ term by term. We divide $L_1$ into two parts.
\begin{align*}
L_1=\langle S({\bf{v}})\partial_t\partial_y{\bf{v}}, \partial_y{\bf{v}}\rangle_{H^{2, 0}}+\langle \partial_yS({\bf{v}})\partial_t{\bf{v}}, \partial_y{\bf{v}}\rangle_{H^{2, 0}}=: L_1^1+L_1^2.
\end{align*}
The first term $L_1^1$ can be treated like $K_1$.
\begin{align*}
L_1^1&=\sum\limits_{|\alpha|=0}^{2}\left\langle S({\bf{v}})\partial_\tau^\alpha\partial_t\partial_y {\bf{v}}, \partial_\tau^\alpha\partial_y {\bf{v}}\right\rangle
+\sum\limits_{|\alpha|=1}^{2}\sum\limits_{|\beta|=1}^{|\alpha|}\left\langle \partial_\tau^\beta S({\bf{v}})\partial_\tau^{\alpha-\beta}\partial_t\partial_y {\bf{v}}, \partial_\tau^\alpha\partial_y {\bf{v}}\right\rangle\notag\\
&=\frac 12\frac d{dt}\sum\limits_{|\alpha|=0}^{2}\left\|\sqrt{S({\bf{v}})}\partial_\tau^\alpha\partial_y {\bf{v}}\right\|_{L^2}^2-\frac 12\sum\limits_{|\alpha|=0}^{2}\int_{\Omega} S({\bf{v}})_t(\partial_\tau^\alpha\partial_y {\bf{v}})^2+\sum\limits_{|\alpha|=1}^{2}\sum\limits_{|\beta|=1}^{|\alpha|}\left\langle \partial_\tau^\beta S({\bf{v}})\partial_\tau^{\alpha-\beta}\partial_t\partial_y {\bf{v}},
\partial_\tau^\alpha\partial_y {\bf{v}}\right\rangle\notag\\
&\ge \frac 12\frac d{dt}\sum\limits_{|\alpha|=0}^{2}\left\|\sqrt{S({\bf{v}})}\partial_\tau^\alpha\partial_y {\bf{v}}\right\|_{L^2}^2-C\|\partial_\tau S({\bf{v}})\|_{L^\infty}\|{\bf{v}}\|_{H^{2, 1}}^2-C\|\partial_\tau\partial_y {\bf{v}}\|_{L_x^\infty L_y^2}\|\partial_\tau S({\bf{v}})\|_{H_x^1 L_y^\infty}\|{\bf{v}}\|_{H^{2, 1}}\notag\\
&\ge \frac 12\frac d{dt}\sum\limits_{|\alpha|=0}^{2}\left\|\sqrt{S({\bf{v}})}\partial_\tau^\alpha\partial_y {\bf{v}}\right\|_{L^2}^2-CD(t)^{\frac12}E(t)-Cf(t)E(t).
\end{align*}
The second part can be estimated straightly.
\begin{align*}
|L_1^2|&\lesssim \|\partial_yS({\bf{v}})\|_{L^\infty}\|\partial_t{\bf{v}}\|_{H^{2, 0}}\|{\bf{v}}\|_{H^{2, 1}}+\|\partial_t{\bf{v}}\|_{L_x^\infty L_y^2}\|\partial_yS({\bf{v}})\|_{H_x^2 L_y^\infty}\|{\bf{v}}\|_{H^{2, 1}}\\
&\le CD(t)^{\frac12}E(t)+Cf(t)E(t).
\end{align*}
As a consequence of the above two inequalities, it yields that
\begin{align}\label{4.26}
L_1\ge \frac 12\frac d{dt}\sum\limits_{|\alpha|=0}^{2}\left\|\sqrt{S({\bf{v}})}\partial_\tau^\alpha\partial_y {\bf{v}}\right\|_{L^2}^2-CD(t)^{\frac12}E(t)-Cf(t)E(t).
\end{align}
Next we estimate $L_2$. With the same procedure as (\ref{3.16}), we derive from integration by parts that
\begin{align}\label{4.27}
|L_2|&\le \left|\left\langle \partial_yA({\bf{v}})\partial_x {\bf{v}}, \partial_y {\bf{v}}\right\rangle_{H^{2, 0}}\right|+\sum\limits_{|\alpha|=0}^{2}\left|\langle A({\bf{v}})\partial_x\partial_\tau^\alpha\partial_y {\bf{v}}, \partial_\tau^\alpha\partial_y {\bf{v}}\rangle\right|\notag\\
&\qquad+\sum\limits_{|\alpha|=0}^{2}\sum\limits_{|\beta|=1}^{|\alpha|}\left|\left\langle\partial_\tau^\beta A({\bf{v}})\partial_\tau^{\alpha+1-\beta} \partial_y{\bf{v}}, \partial_\tau^\alpha\partial_y {\bf{v}}\right\rangle_{L^2}\right|\notag\\
&=\left|\left\langle \partial_yA({\bf{v}})\partial_x {\bf{v}}, \partial_y {\bf{v}}\right\rangle_{H^{2, 0}}\right|+\sum\limits_{|\alpha|=0}^{2}\left|-\frac 12\int_{\Omega} \partial_x A({\bf{v}})(\partial_\tau^\alpha\partial_y {\bf{v}})^2\right|\notag\\
&\qquad+\sum\limits_{|\alpha|=0}^{2}\sum\limits_{|\beta|=1}^{|\alpha|}\left|\left\langle\partial_\tau^\beta A({\bf{v}})\partial_\tau^{\alpha+1-\beta} \partial_y{\bf{v}}, \partial_\tau^\alpha\partial_y {\bf{v}}\right\rangle_{L^2}\right|\notag\\
&\lesssim \|\partial_\tau A({\bf{v}})\|_{L^\infty}\|{\bf{v}}\|_{H^{3, 0}}\|{\bf{v}}\|_{H^{2, 1}}+\|\partial_x {\bf{v}}\|_{L^\infty}\|\partial_y A({\bf{v}})\|_{H^{2, 0}}\|{\bf{v}}\|_{H^{2, 1}}\notag\\
&\qquad+\|\partial_y A({\bf{v}})\|_{L^\infty}\|{\bf{v}}\|_{H^{2, 1}}^2+\|\partial_x\partial_y {\bf{v}}\|_{L_x^\infty L_y^2}\|\partial_\tau A({\bf{v}})\|_{H_x^1 L_y^\infty}\|{\bf{v}}\|_{H^{2, 1}}\notag\\
&\le CD(t)E(t)^{\frac12}+Cf(t)E(t).
\end{align}
For $L_3$, we only need to consider the first component of $\partial_yf({\bf{v}}, \partial_y{\bf{v}})$. Indeed, this term can be controlled as the same same arguments for $J_7$ and we omit the details here.
\begin{align*}
\left|\left\langle \frac{(P-q-R\theta)\theta}{R}\partial_y\tilde{q}\partial_y\tilde{u},\tilde{u}\right\rangle_{H^{2, 1}}\right|\le CD(t)^{\frac12}E(t).
\end{align*}
Notice that the other components of $\partial_yf({\bf{v}}, \partial_y{\bf{v}})$ can be bounded analogically. Thus $L_3$ has the bound.
\begin{align}\label{4.28}
|L_3|\le CD(t)E(t)^{\frac12}.
\end{align}
As for $L_4$, by integration by parts
\begin{align*}
L_4&=\left\langle B({\bf{v}})\partial_y^2 {\bf{v}}, \partial_y^2 {\bf{v}}\right\rangle_{H^{2, 0}}-\left\langle B({\bf{v}})\partial_y^2 {\bf{v}}, \partial_y {\bf{v}}\right\rangle_{H^{2, 0}}|_{y=0}\\
&=\left\langle B({\bf{v}})\partial_\tau^2\partial_y^2 {\bf{v}}, \partial_\tau^2\partial_y^2 {\bf{v}}\right\rangle
+\left\langle \partial_\tau B({\bf{v}})\partial_\tau\partial_y^2 {\bf{v}}, \partial_\tau^2\partial_y^2 {\bf{v}}\right\rangle\\
&\quad+\left\langle \partial_\tau^2 B({\bf{v}})\partial_y^2 {\bf{v}}, \partial_\tau^2\partial_y^2 {\bf{v}}\right\rangle
-\left\langle B({\bf{v}})\partial_y^2 {\bf{v}}, \partial_y {\bf{v}}\right\rangle_{H^{2, 0}}|_{y=0}\\
&=:L_4^1+L_4^2+L_4^3+L_4^4.
\end{align*}
Since $B({\bf{v}})$ is positive definite, then
\begin{align*}
L_4^1\ge C_2\|\partial_y{\bf{v}}\|_{H^{2, 1}}^2.
\end{align*}
The next two terms can be estimated directly by using Gagliardo-Nirenberg inequality,
\begin{align*}
|L_4^2|+|L_4^3|\le CD(t)E(t)^{\frac12}+Cf(t)D(t)^{\frac12}E(t)^{\frac12}.
\end{align*}
The boundary term $L_4^4$ also can be handled as $J_9^{4, 3}$. We omit the detail and only list the estimate.
\begin{align*}
|L_4^4|\le \frac{C}{M^{\frac14}}D(t)+CD(t)E(t)^{\frac14}.
\end{align*}
Thus, we deduce that
\begin{align}\label{4.29}
L_4\ge C_2\|\partial_y {\bf{v}}\|_{H^{2, 1}}^2-CD(t)^{\frac12}E(t)-CD(t)E(t)^{\frac14}-\frac{C}{M^{\frac14}} D(t)-Cf(t)D(t)^{\frac12}E(t)^{\frac12}.
\end{align}
Finally, by (\ref{4.7}) again, the last term also can be bounded by
\begin{align}\label{4.30}
|L_5|\le Cf(t)E(t)^{\frac12}+Cf(t)E(t)+Cf(t)D(t)^{\frac12}E(t)^{\frac12}.
\end{align}
Collecting all the estimates (\ref{4.26})-(\ref{4.30}) together, we conclude that
\begin{align}\label{4.31}
&\frac d{dt}\sum\limits_{|\alpha|=0}^{2}\left(\|\sqrt{G_1}\partial_\tau^\alpha\partial_y\tilde{u}\|_{L^2}^2+\|\sqrt{G_2}\partial_\tau^\alpha\partial_y\tilde{\theta}\|_{L^2}^2+\|\sqrt{G_3}\partial_\tau^\alpha\partial_y\tilde{q}\|_{L^2}^2\right)+C_1\|\partial_y(\tilde{u}, \tilde{\theta}, \tilde{q})\|_{H^{2, 1}}^2\notag\\
&\le CD(t)E(t)^{\frac14}+CD(t)^{\frac 12}E(t)+\frac{C}{M^{\frac14}}D(t)+Cf(t)E(t)^{\frac12}+Cf(t)D(t)^{\frac12}E(t)^{\frac12}.
\end{align}
Consequently, we finish the proof of Proposition \ref{Prop4.5}.
\end{proof}

Plugging Proposition \ref{Prop4.4} and \ref{Prop4.5}, and taking $M$ large enough, we infer that
\begin{align}\label{4.32}
\frac d{dt}&\sum\limits_{\beta=0}^{1}\sum\limits_{|\alpha|=0}^{3-\beta}\left(\|\sqrt{G_1}\partial_\tau^\alpha\partial_y^\beta\tilde{u}\|_{L^2}^2+\|\sqrt{G_2}\partial_\tau^\alpha\partial_y^\beta\tilde{\theta}\|_{L^2}^2+\|\sqrt{G_3}\partial_\tau^\alpha\partial_y^\beta\tilde{q}\|_{L^2}^2\right)+C_0D(t) \notag\\
&\le CD(t)^{\frac 14}E(t)^{\frac54}+CD(t)^{\frac 12}E(t)+CD(t)E(t)^{\frac14}+Cf(t)E(t)^{\frac12}+Cf(t)D(t)^{\frac12}E(t)^{\frac12}.
\end{align}
Thus the proof of Proposition \ref{Prop4.3} is done.

\subsection{Lower bound estimate of lifespan of solutions}
Now we begin to prove the Theorem \ref{Thm2} in this section. The initial condition (\ref{4.3}) together with the coordinate transformation implies that
\begin{align}\label{4.33}
E(0)\le 2\varepsilon^2.
\end{align}
In addition, in view of Proposition \ref{Prop4.3}, by H\"older's inequality and using the smallness of $\varepsilon$, we have
\begin{align}\label{4.34}
&\frac d{dt}\sum\limits_{\beta=0}^{1}\sum\limits_{|\alpha|=0}^{3-\beta}\left(\|\sqrt{G_1}\partial_\tau^\alpha\partial_y^\beta\tilde{u}\|_{L^2}^2+\|\sqrt{G_2}\partial_\tau^\alpha\partial_y^\beta\tilde{\theta}\|_{L^2}^2+\|\sqrt{G_3}\partial_\tau^\alpha\partial_y^\beta\tilde{q}\|_{L^2}^2\right)+\frac{C_0}{2}D(t)\notag\\
\le &CE(t)^{\frac53}+Cf(t)E(t)^{\frac12}.
\end{align}
Suppose that $(0, T_\varepsilon)$ is the maximum interval that the a priori assumption (\ref{4.16}) holds. Then integrate it over $(0, t)$ yields that
\begin{align*}
E(t)&\le E(0)+\int_0^{t} CE(t)^{\frac53}+Cf(t)E(t)^{\frac12} \;dt \\
&\le \left(32C\varepsilon^{\frac43}t+8^\frac12 C\varepsilon^\sigma\int_0^t g(t)\;dt+2\right)\varepsilon^2\\
&\le (32C\varepsilon^{\frac43}t+1+2)\varepsilon^2.
\end{align*}
Here, we used the assumption (\ref{4.2}), (\ref{4.16}) and (\ref{4.33}) in the second line. Consequently, take $T_\varepsilon=\frac{1}{32C}\varepsilon^{-\frac43}$, and we deduce that for any $t\in (0, T_\varepsilon)$,
\begin{align*}
E(t)\le 4\varepsilon^2.
\end{align*}
From which, we close the a priori estimate and finish the proof of Theorem \ref{Thm2}.

\section*{Acknowledgement}
The research of F. Xie's research was supported by National Natural Science Foundation of China No.11831003, Shanghai Science and Technology Innovation Action Plan No. 20JC1413000 and Institute of Modern Analysis-A Frontier Research Center of Shanghai.


\bigskip
\begin{thebibliography}{99} 	

\bibitem{AWXY} R. Alexander, Y.-G. Wang, C. Xu, T. Yang, Well-posedness of the Prandtl equation in Sobolev spaces, J. Amer. Math. Soc., 28 (2015), 745-784.	

\bibitem{CWZ} D. Chen, Y. Wang, Z. Zhang, Well-posedness of the linearized Prandtl equation around a non-monotonic shear flow. Ann. Inst. H. Poincar\'e C Anal. Non Lin\'eaire. 35 (2018), 1119-1142.

\bibitem{CRWZ} D. Chen, S. Ren, Y. Wang, Z Zhang, Long time well-posedness of the MHD boundary layer equation in Sobolev space. Anal. Theory Appl., 36 (2020), 1-18.

\bibitem{DG} H. Dietert, D. Gerard-Varet, Well-posedness of the Prandtl equation without any structural assumption. Ann. PDE 5(1), Art. 8, 51, 2019.

\bibitem{EE} W. E, B. Engquist, Blowup of solutions of the unsteady Prandtl’s equation, Comm. Pure Appl. Math., 50 (1997), 1287-1293.

\bibitem{GD} D. G\'erard-Varet, E. Dormy, On the ill-posedness of the Prandtl equations, J. Amer. Math. Soc., 23 (2010), 591-609.

\bibitem{GM} D. G\'erard-Varet, N. Masmoudi, Well-posedness for the Prandtl system without analyticity or monotonicity, Ann. Sci. \'Ec. Norm. Sup\'er., 48 (2015), 1273-1325.

\bibitem{GN} D. G\'erard-Varet, T. Nguyen, Remarks on the ill-posedness of the Prandtl equation, Asymptot. Anal., 77 (2012), 71-88.

\bibitem{HLY} Y. Huang, C.-J. Liu, T. Yang, Local-in-time well-posedness for compressible MHD boundary layer, J. Differ. Equ., 266 (2019), 2978-3013.

\bibitem{IV} M. Ignatova, V. Vicol, Almost global existence for the Prandtl boundary layer equations, Arch. Ration. Mech. Anal., 220 (2016), 809-848.

\bibitem{LCS} M. Lombardo, M. Cannone, M. Sammartino, Well-posedness of the boundary layer equations, SIAM J. Math. Anal., 35 (2003), 987-1004.

\bibitem{LX} S. Li, F. Xie, Global solvability of 2D MHD boundary layer equations in analytic function spaces. J. Differ. Equ., 299 (2021), 362-401.

\bibitem{LNT} W.-X. Li, N. Masmoudi, T. Yang, Well-posedness in Gevrey function space for 3D Prandtl equations without Structural Assumption. Comm. Pure Appl. Math. doi:10.1002/cpa.21989.

\bibitem{LY2} W.-X. Li, T. Yang, Well-posedness in Gevrey space for the Prandtl equations with non-degenerate points, J. Eur. Math. Soc. (JEMS), 22 (2020), 717-775.

\bibitem{LWXY} C.-J. Liu, D. Wang, F. Xie, T. Yang, Magnetic effects on the solvability of 2D MHD boundary layer equations without resistivity in Sobolev spaces. J. Funct. Anal., 279 (2020), 45 pp.

\bibitem{LWY} C.-J. Liu, Y.-G. Wang, T. Yang, Global existence of weak solutions to the three- dimensional Prandtl equations with a special structure, Discrete Contin. Dyn. Syst. Ser. S, 9 (2016),  2011-2029.

\bibitem{LWY2} C.-J. Liu, Y.-G. Wang, T. Yang, A well-posedness theory for the Prandtl equations in three space variables, Adv. Math., 308 (2017), 1074-1126.
	
\bibitem{LWY3} C.-J. Liu, Y.-G. Wang, T. Yang, On the ill-posedness of the Prandtl equations in three space dimensions, Arch. Ration. Mech. Anal., 220 (2016), 83-108.

\bibitem{LXY2} C.-J. Liu, F. Xie, T. Yang, MHD boundary layers theory in Sobolev spaces without mono- tonicity I: Well-posedness theory, Comm. Pure Appl. Math., 72 (2019), 63-121.

\bibitem{LXY3} C.-J. Liu, F. Xie, T. Yang, Justification of Prandtl ansatz for MHD boundary layer, SIAM J. Math. Anal., 51(2019), 2748-2791.

\bibitem{LY} C.-J. Liu, T. Yang, Ill-posedness of the Prandtl equations in Sobolev spaces around a shear flow with general decay, J. Math. Pures Appl., 108 (2017), 150-162.

\bibitem{LZ3} N. Liu, P. Zhang, Global small analytic solutions of MHD boundary layer equations, J. Differ. Equ., 281 (2021), 199-257.

\bibitem{MW} N. Masmoudi, T. Wong, Local-in-time existence and uniqueness of solutions to the Prandtl equations by energy methods, Comm. Pure Appl. Math., 68 (2015), 1683-1741.

\bibitem{O} O. Oleinik, The Prandtl system of equations in boundary layer theory, Dokl. Akad. Nauk SSSR, 4 (1963), 583-586.

\bibitem{OS}  O. Oleinik, V. N. Samokhin, Mathematical Models in Boundary Layers Theory, Chapman \& Hall/CRC, Boca Raton, FL, 1999.

\bibitem{PZ} M. Paicu, P. Zhang, Global existence and the decay of solutions to the Prandtl system with small analytic data, Arch. Ration. Mech. Anal., 241 (2021), 403-446.

\bibitem{P} L. Prandtl, Uber fl\"ussigkeits-bewegung bei sehr kleiner reibung. Verhandlungen des III, Internationalen Mathematiker Kongresses, Heidelberg, Teubner, Leipzig, 1904, 484-491.

\bibitem{SC} M. Sammartino, R. Caflisch, Zero viscosity limit for analytic solutions, of the Navier-Stokes equation on a half-space. I. Existence for Euler and Prandtl equations, Comm. Math. Phys., 192 (1998), 433-461.

%\bibitem{WW2} C. Wang, Y. Wang, On the hydrostatic approximation of the MHD equations in a thin strip, SIAM J. Math. Anal., 54 (2022), 1241-1269.

\bibitem{XY}  F. Xie, T. Yang, Lifespan of solutions to MHD boundary layer equations with analytic perturbation of general shear flow, Acta Math. Appl. Sin. Engl. Ser., 35 (2019), 209-229.

\bibitem{XZ} Z. Xin and L. Zhang, On the global existence of solutions to the Prandtl system, Adv. Math., 181 (2004), 88-133.

\bibitem{ZZ} P. Zhang, Z. Zhang, Long time well-posedness of Prandtl system with small and analytic initial data, J. Funct. Anal., 270 (2016), 2591-2615.
		
\end{thebibliography}
\end{document}